\documentclass[reqno, 11pt, a4paper]{amsart}
\usepackage{blindtext}

\usepackage{hyperref}

\usepackage{enumerate}
\usepackage{verbatim}
\usepackage{pdfsync}
\usepackage{mathabx}

\usepackage[all]{xy}
\usepackage[usenames,dvipsnames]{xcolor}
\usepackage{tikz}
\usepackage{tikz-cd}

\usepackage{tikz,stackengine}

\usepackage{pgfplots}
\usetikzlibrary{3d}
\usetikzlibrary{shapes,decorations.pathreplacing}
\usetikzlibrary{calc}
\usetikzlibrary{backgrounds}
\usetikzlibrary{automata}
\usetikzlibrary{positioning}
\usetikzlibrary{decorations.markings}
\usetikzlibrary{arrows}
\usetikzlibrary{scopes}
\usetikzlibrary{intersections}
\tikzstyle myBG=[line width=3pt,opacity=1]

\usepackage{tikz}
\usepackage{tikz-cd}
\usetikzlibrary{arrows,automata}
\usetikzlibrary{decorations.markings}
\usetikzlibrary{decorations.markings,arrows}
\usetikzlibrary{arrows,quotes}
\usetikzlibrary{arrows.meta}
\usetikzlibrary{lindenmayersystems,arrows.meta}
\usetikzlibrary{calc}

\tikzset{%
  >={Latex[width=2mm,length=2mm]},
            base/.style = {rectangle, rounded corners, draw=black,
                           minimum width=2cm, minimum height=0.8cm,
                           text centered, font=\sffamily},
  open/.style = {base},
       undec/.style = {base, fill=red!30},
    dec/.style = {base, fill=green!30},
         process/.style = {base, minimum width=2.5cm, fill=orange!15,
                           font=\ttfamily},
}

\newcommand{\lb}{\langle}
\newcommand{\rb}{\rangle}

\newcommand{\R}{\mathfrak{R}}

\newcommand{\gr}{\mathrel{\mathscr{R}}}

\newcommand{\ZM}{{\mathbb{Z}M}}

\newcommand{\ZN}{{\mathbb{Z}N}}
\newcommand{\Z}{\mathbb{Z}}

\newcommand{\FPn}{{\rm FP}\sb n}

\newcommand{\F}{{\rm F}}

\newcommand{\FP}{{\rm FP}}
\newcommand{\FPinfty}{{\rm FP}\sb \infty}

\definecolor{cof}{RGB}{219,144,71}
\definecolor{pur}{RGB}{186,146,162}
\definecolor{greeo}{RGB}{91,173,69}
\definecolor{greet}{RGB}{52,111,72}

\usepackage{amssymb}
\usepackage{amsgen}
\usepackage{amsmath}
\usepackage{amsthm}
\usepackage{mathrsfs}
\usepackage{cite}
\usepackage{amsfonts}

\newcommand{\inv}{^{-1}}

\newcommand{\ov}[1]{\ensuremath{\overline {#1}}}

\newtheorem{Thm}{Theorem}[section]
\newtheorem*{theorema}{Theorem A}
\newtheorem*{theoremb}{Theorem B}

\newtheorem{Prop}[Thm]{Proposition}
\newtheorem{theorem}[Thm]{Theorem}

\newtheorem{Lemma}[Thm]{Lemma}
{\theoremstyle{definition}
}
{\theoremstyle{remark}
\newtheorem{Rmk}[Thm]{Remark}}
{\theoremstyle{remark}
\newtheorem{remark}[Thm]{Remark}}

{\theoremstyle{remark}
}

\newtheorem{lemma}[Thm]{Lemma}
\newtheorem{definition}[Thm]{Definition}

{\theoremstyle{remark}
\newtheorem{Claim}{Claim}}
{\theoremstyle{remark}
}
{\theoremstyle{remark}
}

{\theoremstyle{remark}
}

\numberwithin{equation}{section}

\usepackage{graphicx}

\title[Two-sided homological properties of special monoids]{
Two-sided homological properties of
special and one-relator monoids
}

\keywords{Homological finiteness property,
bi-$\FPn$,
cohomological dimension,
Hochschild cohomological dimension,
monoid,
special monoid,
one-relator monoid.
\\
\indent
This work was supported by the
EPSRC grant
EP/V032003/1 `Algorithmic, topological and geometric aspects of infinite groups, monoids and inverse semigroups'.  The second author was supported by a Simons Foundation Collaboration Grant, award number 849561, the Australian Research Council Grant DP230103184 and Marsden Fund Grant MFP-VUW2411
}

\begin{document}
\maketitle

\begin{center}
ROBERT D. GRAY
\footnote{School of Engineering, Mathematics, and Physics, University of East Anglia, Norwich NR4 7TJ, England.
Email \texttt{Robert.D.Gray@uea.ac.uk}.
}
and
BENJAMIN STEINBERG\footnote{
Department of Mathematics, City College of New York, Convent Avenue at 138th Street, New York, New York 10031,  USA.
Email \texttt{bsteinberg@ccny.cuny.edu}.}
\\
\today
\end{center}

\begin{abstract}
A monoid presentation is called special if the right-hand side of each defining relation is equal to 1. We prove results which relate the two-sided homological finiteness properties of a monoid defined by a special presentation with those of its group of units. Specifically we show that the monoid enjoys the homological finiteness property bi-$\FPn$ if its group of units is of type $\FPn$. We also obtain results which relate the
Hochschild cohomological dimension of the monoid to the cohomological dimension of its group of units. In particular we show that the Hochschild cohomological dimension of the monoid is bounded above by the maximum of $2$ and the cohomological dimension of its group of units. We apply these results to prove    a Lyndon's Identity type theorem for the two-sided homology of one-relator monoids of the form $\lb A \mid r=1 \rb$. In particular, we show that all such monoids are of type bi-$\FPinfty$. Moreover, we show that if $r$ is not a proper power then the one-relator monoid has Hochschild cohomological dimension at most $2$, while if $r$ is a proper power then it has infinite Hochschild cohomological dimension.
For any non-special one-relator monoid $M$ with defining relation $u=v$ we show that if there is no nonempty word $w$ such that $u,v \in A^*w \cap w A^*$ then $M$ is of type bi-$\FP_\infty$ and has Hochschild cohomological dimension at most $2$.   
\end{abstract}

\section{Introduction}
Algorithmic problems in algebra have their origins in logic and topology in the fundamental work of Dehn, Thue and Tietze in the early 1900s; see~\cite{BookAndOtto, LyndonAndSchupp, Margolis:1995qo}. The most important algorithmic question concerning an algebraic structure is the word problem, which asks whether one can decide whether two expressions over generators represent the same element. Markov~\cite{Markov1947} and Post~\cite{Post1947} proved independently that in general the word problem for finitely presented monoids is undecidable. This was later extended to cancellative monoids by Turing~\cite{Turing1950} and then to groups by Novikov~\cite{Novikov1952} and Boone~\cite{Boone1959}.

Given that the word problem is undecidable in general, a central theme running through the development of geometric and combinatorial group and monoid theory has been to identify and study families all of whose members have solvable word problem; see~\cite[Chapter~2]{bridson2002invitations}.
One of the most important and interesting such families is the class of one-relator groups, which were shown to have decidable word problem in classical work of Magnus~\cite{Magnus1932} in the 1930s. 
The theory of one-relator groups has  developed extensively over the last century and continues to be a highly active area of research; see~\cite{linton2025theory} for a recent survey. 

One of the challenges in studying one-relator groups is that their geometry can be quite complicated. 
On the other hand, one-relator groups have turned out to be susceptible for study using topological and homological methods. A classical result of this kind is Lyndon's Identity Theorem~\cite{Lyndon1950} which, interpreted topologically, constructs certain natural classifying spaces for one-relator groups that can be used to deduce all one-relator groups are $\FP_\infty$, and to determine their cohomological dimension.  
Some recent important results on one-relator groups that have been impacted by topological and homological methods include: residual finiteness of one-relator groups with torsion~\cite{wise2021structure}, the proof that one-relator groups with torsion are virtually free-by-cyclic~\cite{kielak2024virtually}, and the coherence of one-relator groups~\cite{jaikin2025coherence}. In particular, a key step in the recent breakthrough coherence theorem of Jaikin-Zapirain and Linton~\cite{jaikin2025coherence} was to establish homological coherence of one-relator groups, meaning that every finitely generated subgroup is of type $\FP_2$.

This paper is part of a programme of research initiated by the authors in 
\cite{GraySteinbergAGT, GraySteinbergDocumenta, GraySteinbergSelecta} aimed at developing topological and homological methods for monoids
and then applying these new results and techniques to the class of one-relator monoids\footnote{That is, monoids defined by presentations with a single defining relation $u=v$ where $u$ and $v$ are words over some finite alphabet $A$}.  
In contrast to one-relator groups, far less is currently known about one-relator monoids, and many fundamental questions about them remain open.  Indeed, it is an important longstanding open problem whether the word problem is decidable for one-relator monoids. While in general this problem is open, it has been solved in a number of key cases in the work of Adjan, Adjan and Oganesyan, and Lallement; see~\cite{Adjan1966,Adyan1987,Lallement1988}.

A natural way to attack the word problem for one-relator monoids is via the 
theory of string rewriting systems. Indeed, an open problem which is intimately related to the word problem for one-relator monoids is the question of whether every one-relator monoid admits a finite complete rewriting system (meaning a monoid presentation which is confluent and terminating; see~\cite[Chapter~12]{HoltBook}). Since complete rewriting systems give computable normal forms, a positive answer to this question would solve the word problem for one-relator monoids.

\begin{figure}
\makebox[\textwidth][c]{
\centering
    \resizebox{1.1\textwidth}{!}{%
\begin{tikzpicture}[node distance=1.8cm,
    every node/.style={minimum height=1.5cm}, align=center, 
implies/.style={double,double equal sign distance,-implies}, 
bendimplies/.style={line width=0.75pt, bend left, double,double equal sign distance,-implies}]
\newcommand\xsep{5}
\newcommand\ysep{3}  
\node[text width=2cm]  (b) at (2.5,3){$\Longrightarrow \cdots \Longrightarrow$};
\node[text width=2cm]  (b) at (7.5,3){$\Longrightarrow \cdots \Longrightarrow$};
\node[text width=2cm]  (b) at (2.5,3-3){$\Longrightarrow \cdots \Longrightarrow$};
\node[text width=2cm]  (b) at (7.5,3-3){$\Longrightarrow \cdots \Longrightarrow$};
\node[text width=2cm]  (b) at (2.5,3-6){$\Longrightarrow \cdots \Longrightarrow$};
\node[text width=2cm]  (b) at (7.5,3-6){$\Longrightarrow \cdots \Longrightarrow$};
\node[text width=2cm]  (b) at (2.5,3-9){$\Longrightarrow \cdots \Longrightarrow$};
\node[text width=2cm]  (b) at (7.5,3-9){$\Longrightarrow \cdots \Longrightarrow$};
\node (58)[open,minimum height=1.5cm,minimum width=2cm]  at (2*\xsep,\ysep){bi-$\FP_3$ \\ ($=$ FHT)};
\node (59)[open,minimum height=1.5cm,minimum width=2cm]  at (2*\xsep,2*\ysep){FDT};

\node (11)[open,minimum height=1.5cm,minimum width=2cm]  at (0,0){left- and \\ right-$\FP_\infty$};
\node (12)[open,minimum height=1.5cm,minimum width=2cm]   at (\xsep,0){left- and \\ right-$\FP_n$};
\node (57)[open,minimum height=1.5cm,minimum width=2cm]   at (2*\xsep,0){left- and \\ right-$\FP_3$};
  
\node (21)[open, minimum height=1.5cm,minimum width=1cm]   at (0,-\ysep){left-$\FP_\infty$};
  \node (new)[open, minimum height=1.5cm,minimum width=2cm]   at (\xsep,-\ysep){left-$\FP_n$};
  \node (22)[open, minimum height=1.5cm,minimum width=2cm]   at (0,\ysep){bi-$\FP_\infty$};
  \node (56)[open,minimum height=1.5cm,minimum width=2cm]   at (2*\xsep,-\ysep){left-$\FP_3$};
  
  \node (31)[open,minimum height=1.5cm,minimum width=2cm]   at (0,-2*\ysep){right-$\FP_\infty$};
  \node (32)[open,minimum height=1.5cm,minimum width=2cm]   at (\xsep,-2*\ysep){right-$\FP_n$};
  \node (55)[open,minimum height=1.5cm,minimum width=2cm]   at (2*\xsep,-2*\ysep){right-$\FP_3$};
  
\node (252)[open, minimum height=1.5cm,minimum width=2cm]   at (\xsep,\ysep){bi-$\FP_n$};

\node (68)[open, minimum height=1.5cm,minimum width=2cm]   at (3*\xsep,\ysep){bi-$\FP_2$};
\node (67)[open, minimum height=1.5cm,minimum width=2cm]   at (3*\xsep,0*\ysep){left- and \\ right-$\FP_2$};
\node (66)[open, minimum height=1.5cm,minimum width=2cm]   at (3*\xsep,-1*\ysep){left-$\FP_2$};
\node (78)[open, minimum height=1.5cm,minimum width=2cm]   at (4*\xsep,\ysep){bi-$\FP_1$ \\ ($=$ finitely \\ dominated)};
\node (77)[open, minimum height=1.5cm,minimum width=2cm]   at (4*\xsep,0*\ysep){left- and \\ right-$\FP_1$};
\node (76)[open, minimum height=1.5cm,minimum width=2cm]   at (4*\xsep,-1*\ysep){left-$\FP_1$ \\ ($=$ f.g.~universal \\ left congruence)};
\node (65)[open, minimum height=1.5cm,minimum width=2cm]   at (3*\xsep,-2*\ysep){right-$\FP_2$};
\node (75)[open, minimum height=1.5cm,minimum width=2cm]   at (4*\xsep,-2*\ysep){right-$\FP_1$ \\ ($=$ f.g.~universal \\ right congruence)};

\node (69)[open, minimum height=1.5cm,minimum width=2cm]   at (3*\xsep,2*\ysep){Finitely \\ presented};
\node (79)[open, minimum height=1.5cm,minimum width=2cm]   at (4*\xsep,2*\ysep){Finitely \\ generated};
  
\node (new0)[open,minimum height=1.5cm,minimum width=3cm]  at (0,2*\ysep){Admits a \\ finite complete \\ rewriting system};


\path (67) edge [bend right=35, line width=0.75pt, double,double equal sign distance,-implies]  node {} (65);
\path (77) edge [bend right=45, line width=0.75pt, double,double equal sign distance,-implies]  node {} (75);
  \draw[implies] ($(new0)!2cm!(59)$) -- ($(59)!2cm!(new0)$);
  \draw[implies] ($(58)!2cm!(59)$) -- ($(59)!2cm!(58)$);
  \draw[implies] ($(57)!2cm!(58)$) -- ($(58)!2cm!(57)$);
  \draw[implies] ($(56)!2cm!(57)$) -- ($(57)!2cm!(56)$);
  \draw[implies] ($(68)!2cm!(69)$) -- ($(69)!2cm!(68)$);
  \draw[implies] ($(67)!2cm!(68)$) -- ($(68)!2cm!(67)$);
  \draw[implies] ($(66)!2cm!(67)$) -- ($(67)!2cm!(66)$);
  \draw[implies] ($(78)!2cm!(79)$) -- ($(79)!2cm!(78)$);
  \draw[implies] ($(77)!2cm!(78)$) -- ($(78)!2cm!(77)$);
  \draw[implies] ($(76)!2cm!(77)$) -- ($(77)!2cm!(76)$);
  \draw[implies] ($(22)!2cm!(new0)$) -- ($(new0)!2cm!(22)$) ;
  \draw[implies] ($(11)!2cm!(22)$) -- ($(22)!2cm!(11)$);
  \draw[implies] ($(21)!2cm!(11)$) -- ($(11)!2cm!(21)$);
  \draw[implies] ($(55)!2cm!(65)$) -- ($(65)!2cm!(55)$);
  \draw[implies] ($(65)!2cm!(75)$) -- ($(75)!2cm!(65)$);
  \draw[implies] ($(59)!2cm!(69)$) -- ($(69)!2cm!(59)$);
  \draw[implies] ($(69)!2cm!(79)$) -- ($(79)!2cm!(69)$);
  \draw[implies] ($(58)!2cm!(68)$) -- ($(68)!2cm!(58)$);
  \draw[implies] ($(68)!2cm!(78)$) -- ($(78)!2cm!(68)$);
  \draw[implies] ($(57)!2cm!(67)$) -- ($(67)!2cm!(57)$);
  \draw[implies] ($(67)!2cm!(77)$) -- ($(77)!2cm!(67)$);
  \draw[implies] ($(56)!2cm!(66)$) -- ($(66)!2cm!(56)$);
  \draw[implies] ($(66)!2cm!(76)$) -- ($(76)!2cm!(66)$);
  \draw[implies] ($(12)!2cm!(252)$) -- ($(252)!2cm!(12)$);
  \draw[implies] ($(new)!2cm!(12)$) -- ($(12)!2cm!(new)$);
\path (11) edge [bend right=35, line width=0.75pt, double,double equal sign distance,-implies]  node {} (31);
\path (12) edge [bend right=35, line width=0.75pt, double,double equal sign distance,-implies]  node {} (32);
\path (57) edge [bend right=35, line width=0.75pt, double,double equal sign distance,-implies]  node {} (55);
\end{tikzpicture}
}}
\label{fig_FCRSHomologicalProperties}
\caption{
Homological finiteness properties satisfied by monoids that admit finite complete rewriting systems, and all implications between them.  
The proofs that all these implications hold, 
and that none of the implications in the diagram are reversible, 
can be found in the papers~\cite{cremanns1994finite, lafont1995new, Pride1995, wang2000second, KobayashiOtto2003, Kobayashi2005, Pride2006, Kobayashi2010}. 
\vspace{-6mm}
}
\end{figure}
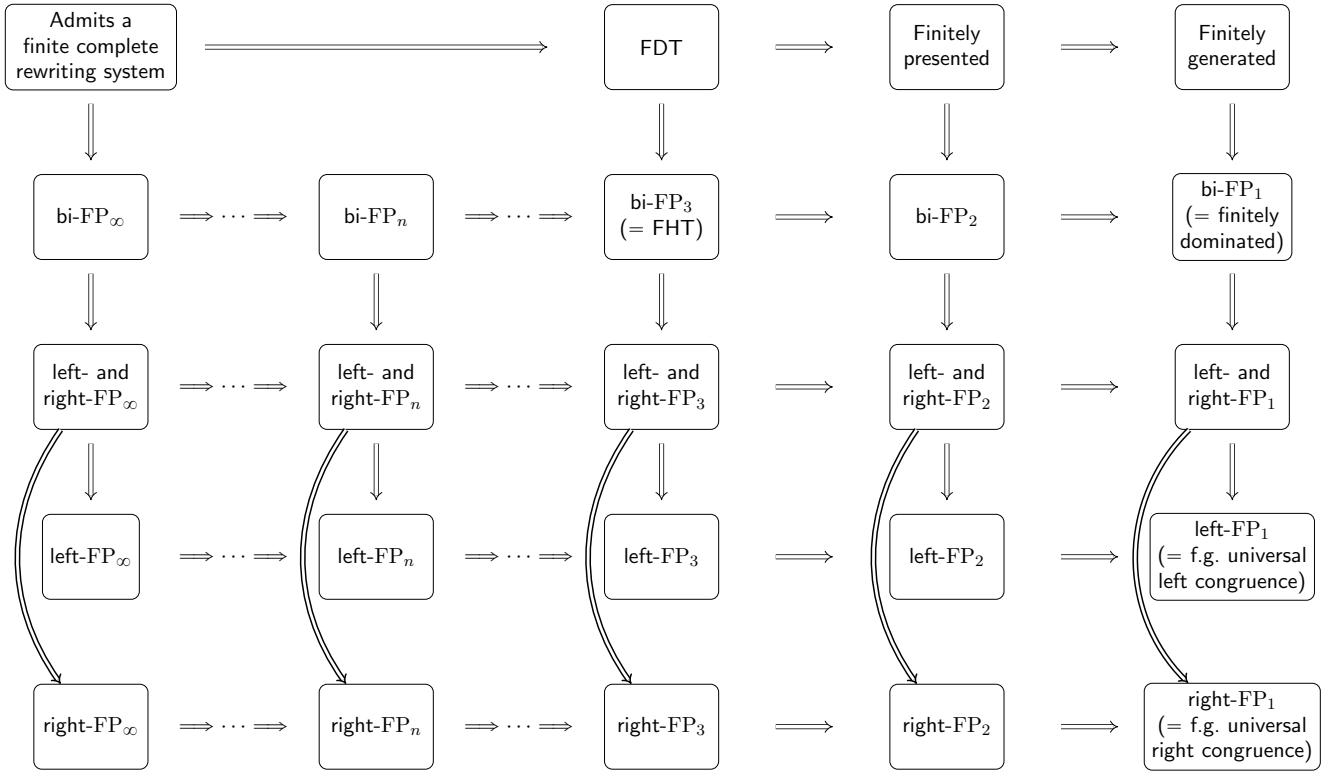

Homological and topological methods provide important tools for investigating the question of whether a given monoid admits a presentation by a finite complete rewriting system. The Anick-Groves-Squier theorem, cf.~\cite{Brown1989}, shows that if a monoid admits a finite complete rewriting system, then the monoid must satisfy the homological finiteness properties left- and right-$\FP_\infty$. Kobayashi~\cite{Kobayashi2005} improved this result by showing a monoid admitting a finite complete rewriting is of type bi-$\FPinfty$. There are in fact a range of finiteness properties satisfied by monoids that admit finite complete rewriting systems, and they have received extensive attention in the literature. The homotopical finiteness property FDT (finite derivation type) was introduced by Squier in~\cite{Squier1994} where it was shown that any monoid admitting a finite complete rewriting system had FDT. The two-sided homological finiteness property bi-$\FP_n$ was introduced in~\cite{KobayashiOtto2001}, and in~\cite{KobayashiOtto2003} it was proved that bi-$\FP_3$ is equivalent to the property FHT (finite homological type). Note that Alonso and Hermiller in~\cite{AlonsoHermiller2003} introduced another homological finiteness property satisfied by monoids admitting finite complete rewriting systems (that they also called bi-$\FP_n$) that was later shown by Pride~\cite{Pride2006} to be equivalent to a monoid satisfying both left-$\FP_n$ and right-$\FP_n$. 

In Figure~\ref{fig_FCRSHomologicalProperties} we have arranged all these properties and the various implications between them. Note that by results in~\cite{cremanns1994finite, lafont1995new, Pride1995, wang2000second, KobayashiOtto2003, Kobayashi2005, Pride2006, Kobayashi2010} 
none of the implications in Figure~\ref{fig_FCRSHomologicalProperties} are reversible in general. In particular, in relation to the results we prove in this paper, it is important to stress that it is \emph{not} true that a monoid satisfying left-$\FP_n$ and right-$\FP_n$ is of type bi-$\FP_n$.

These homological finiteness properties have also arisen naturally in connection to other areas including: monoids with finitely generated universal left congruence, which is equivalent to left-$\FP_1$ (see~\cite{dandan2019semigroups, east2024diameter,kobayashi2007homological}),  
the notion of ancestry in a monoid, due to White from the study of Banach algebras on groups and monoids see~\cite[Section 3]{dandan2019semigroups}  
and~\cite{White2017}, and finitely dominated monoids (meaning there is a finite subset that dominates the monoid in the sense of  Howie and Isbell~\cite{howie1967epimorphisms}),
which is equivalent to bi-$\FP_1$; see~\cite{kobayashi2007homological}. 
Monoids of type left-$\FP_n$ are also used to define the higher dimensional BNSR invariants of groups; see~\cite{Bieri1987, Bieri1988}.

Motivated by the open question of whether one-relator monoids admit finite complete rewriting systems, together with the Anick-Groves-Squier theorem, Kobayashi initiated the study of homological finiteness properties of one-relator monoids in the papers~\cite{kobayashi1998homotopy, Kobayashi2000}. 
In these papers 
Kobayashi resolved this problem in low dimensions by showing that all one-relator monoids are 
FDT and thus are of type left-$\FP_3$ and right-$\FP_3$. 
He then asked the natural question~\cite[Problem~1]{Kobayashi2000} of whether this result could be extended to higher dimension left- and right-$\FP_n$ for all $n$.    
In~\cite{GraySteinbergSelecta} we gave a positive answer to Kobayashi's problem by proving that every one-relator monoid is of type left-$\FP_\infty$ and right-$\FP_\infty$. Looking at the diagram in Figure~\ref{fig_FCRSHomologicalProperties} it is clear that the remaining main open question in this area is to determine whether all one-relator monoids are of type bi-$\FP_\infty$. 
Establishing the property bi-$\FP_\infty$ is the best possible result 
on homological finiteness properties  
that one could hope to prove for one-relator monoids since it would imply all other previous results on such properties for 
this class.

One of the main applications of the general results we prove in this paper is to resolve this problem in the key case of one-relator monoids where the defining relation is of the form $r=1$.  
We obtain results in this paper on one-relator monoids of the form $\lb A \mid r=1 \rb$ by proving more general results for monoids defined by finite presentations of the form $\langle A\mid w_1=1,\ldots,w_k=1\rangle$. These are called \emph{special monoids} in the literature. Special monoids were studied in the sixties by Adjan~\cite{Adjan1966} and Makanin~\cite{Makanin66}. Adjan proved that the group of units of a one-relator special monoid is a one-relator group and solved the word problem for this class of one-relator monoids by reducing the word problem of the monoid to that of the group of units. Makanin proved more generally that the group of units of a $k$-relator special monoid is a $k$-relator group and showed that the monoid has decidable word problem if and only if its group of units does. Later Zhang~\cite{Zhang} gave a more modern approach to these results via the theory of string rewriting systems. 

The main general result in this paper reduces the problem of whether the homological finiteness property bi-$\FP_n$ holds in a special monoid, to the corresponding question for its group of units. In addition, our methods can be applied to relate the Hochschild cohomological dimension of the monoid to the cohomological dimension of its group of units. 
Throughout we shall use $\mathrm{cd}(G)$ to denote the cohomological dimension of a group $G$. 
Our first main result is:   

\begin{theorema}\label{thm:TheoremA}
Let $M$ be a finitely presented monoid of the form 
\[M = \lb A \mid w_1=1, \ldots, w_k=1 \rb\]
and let $G$ be the group of units of $M$.
\begin{enumerate}
\item[(i)]
If $G$ is of type $\FPn$ (for $1 \leq n \leq \infty$) then $M$ is of type bi-$
\FPn$.
\item[(ii)]
Furthermore, the Hochschild cohomological dimension of $M$ is bounded below by
$\mathrm{cd}(G)$
and bounded above by
$\mathrm{max}\{ 2, \mathrm{cd}(G) \}$.
  \end{enumerate}
\end{theorema}
We shall then go on to apply this result to prove the following two-sided Lyndon's Identity Theorem for one-relator monoids of the form $\langle A\mid w=1\rangle$.

\begin{theoremb}\label{TheoremB}
Every one-relator monoid of the form $\langle A \mid r=1 \rangle$ is of type bi-$\FP_\infty$. Moreover, if $r$ is not a proper power then the one-relator monoid has Hochschild cohomological dimension at most $2$, while if $r$ is a proper power then the monoid has infinite Hochschild cohomological dimension.
  \end{theoremb}
As explained  in Section~\ref{sec:other} below, Theorem~B gives a positive solution for an important infinite family of base cases towards a proof of the more general conjecture that all one-relator monoids $\langle A\mid u=v \rangle$ are of type bi-$\FP_\infty$.

The property bi-$\FP_n$ is significantly more robust, and carries more information about a monoid, than either of the properties left-$\FP_n$ or right-$\FP_n$. For instance, for any monoid $M$ if we adjoin a zero element to form $M^0 = M \cup \{0 \}$, then the monoid $M^0$ will automatically be of type left- and right-$\FP_\infty$ while in contrast for any integer $n$ the monoid $M^0$ will be of type bi-$\FP_n$ if and only if $M$ itself is of type bi-$\FP_n$; see~\cite[Section~3]{GraySteinbergDocumenta}. There are other results that provide evidence that bi-$\FP_n$ is a more reasonably behaved property than its one-sided counterparts. 
For instance, in  
\cite{GraySteinbergInPrepr2025} we shall prove that 
a von Neumann  regular monoid with finitely many left and right ideals is of type bi-$\FP_n$ if and only if all of its maximal subgroups are of type $\FP_n$. 
For left- or right-$\FP_n$ this result fails to hold since any group with a zero element adjoined is left- and right-$\FP_\infty$.

The importance of Hochschild cohomology stems from the fact that it is the standard notion of cohomology for rings;~\cite{HochschildCoh},\cite[Chapter~9]{Weibel1994}, or~\cite{Mitchell1972}. See~\cite{witherspoon2020hochschild} for a recent survey article on Hochschild cohomology and its many important uses in algebra.  
The Hochschild cohomological dimension of a monoid bounds both its left and right cohomological dimension, but there are examples where all three dimensions are different; see~\cite{Guba1998, cheng1984separable}.

To prove our main results we shall need to introduce and develop an entirely new two-sided analogue of the classical Adjan--Makanin--Zhang theory of special monoids. 
This makes the results in this paper significantly more difficult than the corresponding one-sided results proved in~\cite[Section~3]{GraySteinbergDocumenta}. 
To give an example of this, implicit in the work of Adjan and Makanin, and 
explicit in Zhang, is the fundamental result that the 
submonoid of 
right units (dually left units) of a special monoid $M$ is isomorphic to the free product of the group of units and a finite rank free monoid. 
This algebraic description of the right units is key to relating the structure of special monoids and their groups of units. For the two-sided theory developed in this paper the right units $R$ and left units $L$ are replaced by a single monoid 
$N = \{ (x,y) \in R \times L: xy=1 \}$, which is a submonoid of $M\times M^{op}$. 
Here $M^{op}$ denotes the opposite of the monoid $M$ 
(see Section~\ref{sec:prel} below for the definition of this notion). 
It turns out, rather miraculously, that for any special monoid $M$ this submonoid $N$ of $M\times M^{op}$ is also isomorphic to a free product of the group of units of $M$ with a finite rank free monoid.

Another key ingredient for studying special monoids is the fact, proved in~\cite[Corollary 3.12]{GraySteinbergDocumenta}, that any two right Sch\"{u}tzenbeger graphs of a special monoid are isomorphic to each other (as directed labelled graph) and that modulo these graphs the Cayley graph has a tree-like structure. 
This structural description of the right Cayley graph together with results about the left-multiplication action of $M$ on it were then combined in~\cite[Theorem 3.20]{GraySteinbergDocumenta} to build a left equivalent classifying space for the monoid.
To prove the two-sided results in this paper we need to establish several new results 
that 
give a sufficient level of information about the structure of the two-sided Cayley graph of the monoid and the action of $M\times M^{op}$ upon it.   

In this way, the results in this paper make new contributions to the theory of special monoids by fully developing, for the first time, the two-sided theory. We note in passing that the success of the theory of special monoids has inspired the development of a parallel theory of special inverse monoids initiated in important work of Ivanov, Margolis and Meakin~\cite{ivanov2001one} and then developed further in several recent papers including~\cite{gray2020undecidability, gray2025maximal, gray2024groups}.

The paper is organised as follows. In Section~\ref{sec:prel} we give some preliminary definitions and results on presentations, rewriting systems and homological algebra. In Section~\ref{sec:2sided} we 
give some background  on the theory of finitely presented special monoids including results that give normal forms for the elements of the monoid. Section~\ref{sec:FreeProd:special} is devoted to proving that the monoid $N$, discussed above, has the structure of the free product of a group and a free monoid. In Section~\ref{sec:MxMFreeNSet:special} we prove that $M \times M$ is a free $N$-set and apply this in Lemma~\ref{lem:forst:quotient2} to show that the quotient $\overleftrightarrow{\Gamma}/N$ of the two-sided Cayley graph $\overleftrightarrow{\Gamma}$ of $M$ by $N$ is a forest. Section~\ref{sec:edge:module} is devoted to the proof that the set $E$ of edges collapsed when forming $\overleftrightarrow{\Gamma}/N$ is a finitely generated free $M \times M^{\mathrm{op}}$-set.
In Section~\ref{sec:main:theorem:special} we combine all the results from Sections~\ref{sec:FreeProd:special}--\ref{sec:edge:module} together to prove our main results Theorem~A and Theorem~B. Finally in Section~\ref{sec:other} we give a framework for attacking the problem of proving that all one-relator monoids $\lb A \mid u=v \rb$, explain how Theorem~B deals with an infinite family of base cases for that approach, and also show how to deal with the other infinite family of base cases by showing that if a non-special one-relator monoid $M$ has defining relation $u=v$ such that there is no nonempty word $r$ with $u,v \in A^*r \cap r A^*$, then $M$ is of type bi-$\FP_\infty$ and the Hochschild cohomological dimension at most $2$.

\section{Preliminaries}\label{sec:prel}

\subsection*{Presentations and complete rewriting systems} 

We recall some basic definitions and notation for monoid presentations and rewriting systems. For further details we refer the reader to~\cite[Chapter~12]{HoltBook}. 
Let $A$ be a nonempty set, that we call an alphabet, and let $A^*$ denote the free monoid of all words over $A$. Given two words $u,v \in A^*$ we write $u \equiv v$ to indicate that $u$ and $v$ are equal as words in $A^*$, in which case we say the words are \emph{graphically equal}.
A \emph{rewriting system} over $A$ is a subset $R$ of $A^* \times A^*$. The pair $\lb A \mid R \rb$ is called a \emph{monoid presentation}. The elements of $R$ are called \emph{rewrite rules} of the system, and the \emph{defining relations} of the presentation.  
We write $u \rightarrow_R v$ for words $u,v \in A^*$ if there are words $\alpha, \beta \in A^*$ and a rewrite rule $(\ell,r) \in R $ such that $u \equiv \alpha \ell \beta$ and $v \equiv \alpha r \beta$. We use $\rightarrow_R^*$ to denote the reflexive transitive closure of $\rightarrow_R$, and write $\leftrightarrow_R^*$ for the symmetric closure of $\rightarrow_R^*$. 
When the set of rewrite rules with respect to which we are working with is clear from context, we shall often omit the subscript $R$ and simply write $\rightarrow$, $\rightarrow^*$ and $\leftrightarrow^*$.
The relation $\leftrightarrow_R^*$ is a congruence on the free monoid $A^*$ and the quotient $A^* /  {\leftrightarrow_R^*}$ is the monoid defined by the presentation $\lb A \mid R \rb$. 
When working with a monoid $M$ defined by presentation $\lb A \mid R \rb$, given two words $u,v \in A^*$ we write $u=v$ (or sometimes we write $u =_M v$) to mean they are equal in $M$.

Given a rewriting system $R$ over an alphabet $A$, we say that a word $u \in A^*$ is \emph{irreducible} (or \emph{reduced}) if no rewrite rule can be applied to it, that is, there is no word $v$ such that $u \rightarrow v$. 
The rewriting system $\R$ is \emph{Noetherian} if there is no infinite chain of words $u_i \in A^*$ with $u_i \rightarrow u_{i+1}$ for all $i \geq 1$. We say that the rewriting system is \emph{confluent} if whenever $u \rightarrow^* u_1$ and $u \rightarrow^* u_2$ there is a word $v \in A^*$ such that $u_1 \rightarrow^* v$ and $u_2 \rightarrow^* v$. A rewriting system that is both Noetherian and confluent is called \emph{complete}. If $R$ is a complete rewriting system then each $\leftrightarrow^*$ equivalence class contains a unique irreducible word, so in this situation the irreducible words give a set of normal forms for the elements of the monoid that is defined by the presentation $\lb A \mid R \rb$.

\subsection*{The category of $M$-sets}
For any monoid $M$ a \emph{left $M$-set} consists of a set $X$ and a mapping 
$M \times X \rightarrow X$ written $(m,x) \mapsto mx$ called a \emph{left action}, such that $1x=x$ and $m(nx) = (mn)x$ for all $m,n \in M$ and $x \in X$. 
We define right $M$-sets dually. Note that right $M$-sets are the same thing as left $M^{op}$-sets, where $M^{op}$ is the \emph{opposite} of the monoid $M$ which is the monoid with the same underlying set $M$ and multiplication given by $x \cdot y = yx$.  An \emph{$M$-biset} is an $M\times M^{op}$-set.  We often write $(m_1,m_2)x=m_1xm_2$ in this case. 
A mapping $f\colon X \rightarrow Y$ between $M$-sets is said to be \emph{$M$-equivariant} if $f(mx) = mf(x)$ for all $x \in X$, $m \in M$. For a fixed monoid $M$ the collection of $M$-sets together with $M$-equivariant mappings form a category.

A left $M$-set $X$ is a \emph{free left $M$-set on $B$} if there is a subset $B \subseteq X$ such that each element of $X$ can be expressed uniquely as $mb$ with $m\in M$ and $b\in B$. 
In this situation we call $B$ a \emph{basis for $X$}. 
Free left $M$-sets have the usual universal property.  
The free left $M$-set on $B$ exists and can be realised as the set $M \times B$ with action $m(m',b) = (mm',b)$. Free right $M$-sets are defined dually.  If $X$ is a free left $M$-set with basis $B$, then $\mathbb ZX$ is a free left $\mathbb ZM$-module with basis $B$.

Given any right $M$-set $X$ we define a natural preorder relation $\leq$ on $X$ where $x \leq y$ if and only if $xM \subseteq yM$. We write $x \approx y$ if there is a sequence $z_1, z_2, \ldots, z_n$ of elements of $X$ such that for each $0 \leq i \leq n-1$ either $z_i \leq z_{i+1}$ or $z_i \geq z_{i+1}$. This is an equivalence relation and we call the $\approx$-classes of $X$ the \emph{weak orbits} of the $M$-set. This corresponds to the notion of the weakly connected components in a directed graph. If $X$ is a right $M$-set then we use $X/M$ to denote the set of weak orbits of the $M$-set. 

\subsection*{Homological algebra}
Let $R$ be a ring. Then a left $R$-module $V$ is said to be of \emph{type $\mathrm{FP}_n$} if it has a projective resolution $P_\bullet\to V\to 0$ such that $P_i$ is finitely generated for $0\leq i\leq n$. This is equivalent to having a free resolution that is finitely generated through degree $n$; see~\cite[Proposition 4.3]{BrownCohomologyBook}. We say that $V$ is of type $\FP_\infty$ if it has a projective (equivalently, free) resolution that is finitely generated in all degrees.   
A right $R$-module is $\mathrm{FP}_n$ if it is $\mathrm{FP}_n$ as a left $R^{op}$-module, 
and we say an $R$-$R$-bimodule (called here an $R$-bimodule) is $\mathrm{FP}_n$ if it is $\mathrm{FP}_n$ as a left $R^e=R\otimes_{\mathbb Z}R^{op}$-module.  Note that if $M$ is a monoid, then $\mathbb ZM^e\cong \mathbb Z[M\times M^{op}]$.  We also write $\mathrm{FP}_n(R)$ to indicate that we are referring to the ring $R$ in the case that ambiguity might arise.  
For convenience, all modules will be considered $\mathrm{FP}_{-1}$.
A ring $R$ is of type \emph{bi-$\mathrm{FP}_n$} if the $R$-bimodule $R$ is $\mathrm{FP}_n(R^e)$.
We say that a left $R$-module $V$ has \emph{projective dimension} at most $d$ if it has a projective resolution of length $d$.

A monoid $M$ is of type left-$\mathrm{FP}_n$ (respectively, right-$\mathrm{FP}_n$) if the trivial left (respectively, right) $\ZM$-module $\mathbb Z$ is of type $\mathrm{FP}_n$.  Kobayashi and Otto~\cite{KobayashiOtto2001} defined $M$ to be of type bi-$\mathrm{FP}_n$ if $\mathbb ZM$ is of type bi-$\mathrm{FP}_n$. So, 
the monoid $M$ is of type \emph{bi-$\FPn$}
if there is a projective resolution
\[
 \cdots \rightarrow P_1 \rightarrow
P_0 \rightarrow \mathbb ZM \rightarrow 0
\]
of the $\mathbb ZM$-bimodule $\mathbb ZM$, where $P_0, P_1, \ldots, P_n$ are finitely generated projective $\mathbb{Z}M$-bimodules. 
If a monoid is of type bi-$\FP_n$ then it is both of type left-$\FP_n$ and right-$\FP_n$; see 
\cite[Proposition 4.2]{KobayashiOtto2003}.

Kobayashi~\cite{Kobayashi2005} proved that if a monoid admits a presentation by a finite complete rewriting system then that monoid must be of type bi-$\FP_\infty$. See also~\cite[Section~11]{GraySteinbergAGT} for a topological proof of this using the theory of equivariant collapsing schemes (i.e., discrete Morse theory). 

The \emph{left (right) cohomological dimension} of a monoid $M$ is the projective dimension of the trivial left (right) $\ZM$-module $\Z$.
For a group $G$ the right cohomological dimension is equal to the  
left cohomological dimension and is just called the \emph{cohomological dimension of the group} and is denoted $\mathrm{cd}(G)$. 
The \emph{Hochschild cohomological dimension} of a monoid $M$, written $\dim M$, is the projective dimension of $\mathbb ZM$ as a $\mathbb Z[M\times M^{op}]$-module. The Hochschild cohomological dimension bounds both the left and right cohomological dimension of the monoid.  

We shall need the following result on modules of type $\mathrm{FP}_n$ which may be found in  
\cite[Proposition~1.4]{Bieri1976}.

\begin{Lemma}\label{t:bieri}
Let $R$ be a ring.  If $0\to A\to B\to C\to 0$ is an exact sequence of left $R$-modules such that 
$A$ is of type $\mathrm{FP}_{n-1}$ and 
$B$ is of type $\mathrm{FP}_n$, then $C$ is of type $\mathrm{FP}_n$.
\end{Lemma}

We shall also make use of following result about projective dimension of $R$-modules. The following lemma can be proved by applying~\cite[Lemma~1.5]{Brown1982}; see~\cite[Corollary 2.4]{GraySteinbergDocumenta}. 

\begin{Lemma}\label{c:fp.resolved}
Let $R$ be a ring and suppose that  
\[C_n\longrightarrow C_{n-1}\longrightarrow\cdots\longrightarrow C_0\longrightarrow V\to 0\] is a partial resolution of an $R$-module $V$.
Let $d\geq n$ and suppose that $C_n\to C_{n-1}$ is injective. If $C_i$ has a projective dimension of at most $d-i$, for $0\leq i\leq n$, then $V$ has a projective dimension at most $d$.
\end{Lemma}
In~\cite[Section~3]{GraySteinbergDocumenta} a key part of proving results relating the one-sided homological finiteness properties of a special monoid and those of its group of units involved analysing properties of the right and left Cayley graphs of the monoid.   
In this paper, to study bi-$\FP_n$ and the Hochschild cohomological dimension we shall need to prove results about the structure of the \emph{two-sided Cayley graph} of the monoid $M$ which is defined as follows.     
Let $M$ be a monoid generated by a finite set $A$.   
Then the 
\emph{two-sided Cayley graph 
$\overleftrightarrow{\Gamma} = \overleftrightarrow{\Gamma}(M,A)$ of $M$ with respect to $A$}    
is the directed labelled graph with vertex set $M \times M$ and directed edges $M \times A \times M$ where $(m_1, a, m_2)$ is a directed edge with initial vertex $(m_1, am_2)$ and terminal vertex $(m_1a, m_2)$.      
Just as a monoid acts naturally by left multiplication on its right Cayley graph, there is a natural action of $M \times M^{\mathrm{op}}$ on 
$\overleftrightarrow{\Gamma} = \overleftrightarrow{\Gamma}(M,A)$ 
via
\begin{gather*}
(m_1,m_2)\cdot (x,y) = (m_1x,ym_2)\\
 (m_1, m_2) \cdot (x,a,y) = (m_1 x, a, y m_2)
 \end{gather*}
where $(m_1, m_2) \in M \times M^{\mathrm{op}}$, $x,y \in M$ and  
$a \in A$.  These are free $M \times M^{\mathrm{op}}$-sets with bases $\{(1,1)\}$ and $\{1\}\times A\times \{1\}$, respectively.
The two-sided Cayley graph and this action of $M \times M^{\mathrm{op}}$ will play a central role in the proof of the mains result of this paper.   
See~\cite[Section~5.3]{GraySteinbergAGT} for more on the two-sided Cayley graph and its role in the study of two-sided homological and geometric finiteness properties of monoids.

\section{Finitely presented special monoids}\label{sec:2sided}

We now give some definitions and results due to Zhang \cite{Zhang} that will be used throughout the paper. Let $M = \lb A \mid w_1=1, \ldots, w_k=1 \rb$ and let $G$ be the group of units of $M$. By~\cite[Theorem~3.7]{Zhang}, $G$ has a group
presentation with $k$ defining relations. We call a word $u \in A^*$ \emph{invertible} if it represents an element of the group of units $G$. We say that an invertible word $u \in A^*$ is \emph{indecomposable} if it has no proper prefix that is invertible.  It is straightforward to see that every nonempty invertible word $u$ has a unique factorisation $u \equiv u_1 \ldots u_l$ where each $u_j$ is nonempty and indecomposable. For each relation word $w_i$ in the presentation of $M$ let $w_i \equiv w_{i,1} \ldots w_{i,n_i}$ be its decomposition into nonempty indecomposable invertible words. Given any invertible word $\beta \in A^+$ we call $\beta$ \emph{minimal} if it is indecomposable, and in addition $|\beta| \leq \mathrm{max}_{1 \leq i \leq k}|w_i|$. Now let $\Delta$ be the set of all minimal invertible words $\delta \in A^*$ such that $\delta = w_{i,j}$ in $M$ for some $1 \leq i \leq n, 1 \leq j \leq n_j$. It is immediate from the definition that $\Delta$ is a finite set of indecomposable invertible words. It is proved in~\cite[Section~3]{Zhang} that the elements represented by words from $\Delta$ give a finite generating set for the group of units $G$ of $M$.  
  
An important result of Zhang~\cite{Zhang} shows how $\Delta$ can be used to give an infinite rewriting system for the monoid $M$ modulo the group of units of $M$, in the following way.     
Let $<_s$ be the short-lex order on $A^*$ induced by some fixed order on the generators $A$. Then define a rewriting system
\[
S = \{ u \rightarrow v \mid u, v \in \Delta^*, u = v \ \mbox{in $M$} \ \mbox{and} \ u >_s v \}.
\]
It follows from \cite[Proposition 3.2]{Zhang} that $S$ is an infinite complete rewriting system defining the special monoid $M$.
We call a word $\alpha \in A^*$ \emph{irreducible} (or \emph{reduced}) if it is irreducible with respect to the rewriting system $S$. 
For any word $w \in A$ we shall
use $\overline{w}$ to denote the unique reduced word such that $w = \overline{w}$ in $M$.

The following lemma is folklore.

\begin{lemma}\label{lem:unions:of:invertible:general}
Let $T$ be a monoid and $r,s,t\in T$ with $rs$ and $st$ invertible.  Then $s$ and $rst$ are invertible.
  \end{lemma}
\begin{proof}
First note that $rst[(st)^{-1}s(rs)^{-1}] = 1$ and $[(st)^{-1}s(rs)^{-1}]rst = 1$, whence $rst$ is invertible.  Since $rs,st$ are invertible, $s$ is left and right invertible, hence invertible.
\end{proof}

A key result for us is the following. 

\begin{lemma}\cite[Lemma~3.4]{Zhang}\label{lem:Zhang:Irred}
An irreducible word $u \in A^*$ is invertible if and only if
$u \in \Delta^*$.
\end{lemma}

Another result that will prove useful later on is the following normal form result of Otto and Zhang~\cite[Theorem 5.2 and Lemma~5.4]{otto1991decision}.  We shall say a word $w\in A^*$ is 
\emph{trivial} if it represents the trivial element of $M$. 
We say a subword $u$ of a word $w$ is a \emph{maximal invertible subword} if $u$ is an invertible word and $u$ is not properly contained in some other invertible subword of $w$.
Lemma~\ref{lem:unions:of:invertible:general} implies that any nonempty invertible subword $u$ of a word $w$ is contained in a unique maximal invertible subword of $w$.

\begin{lemma}\cite[Theorem 5.2]{otto1991decision}\label{lem:OttoZhang}
Let $M = \lb A \mid w_1=1, \ldots, w_k=1 \rb$.  Then every word $u \in A^*$ can be uniquely factorised as $u\equiv u_0a_1u_1 \ldots a_m u_m$ where $a_i \in A$, for $1 \leq i \leq m$, and $u_j$ is a maximal invertible subword of $u$ for $0 \leq j \leq m$.  
We call this the Otto-Zhang normal form of $u$.  Moreover, if $v \in A^*$ with $u=v$ in $M$, then the Otto-Zhang normal form of $v$ is $v \equiv v_0a_1v_1 \ldots a_m v_m$ where
 $u_i = v_i$ in $M$ for $i=0,\ldots, m$.
  \end{lemma}

 We note that some of the $u_i$ and $v_j$ in the lemma above may be empty.

\begin{lemma}\label{lemma:OZ} Let $M = \lb A \mid w_1=1, \ldots, w_k=1 \rb$, let $u \in A^*$ and let  
  \[u\equiv u_0a_1u_1 \ldots a_m u_m \]
be the Otto-Zhang normal form of $u$. Then $\ov u\equiv \ov{u_0}a_1\ov{u_1}\cdots a_m\ov{u_m}$ and each $\ov{u_i}\in \Delta^*$. Moreover, this is the Otto-Zhang normal form of $\ov u$.
\end{lemma}
\begin{proof} 
To see this, it suffices to show that each $\ov{u_i}$ is a maximal invertible subword of $v \equiv \ov{u_0}a_1\ov{u_1}\cdots a_m\ov{u_m}$, and that $v$ is reduced. Suppose that $w$ is a maximal invertible subword of $v$. Since each $\ov {u_j}$ is invertible, it follows from Lemma~\ref{lem:unions:of:invertible:general} that if $w$ contains a letter from some $\ov {u_j}$ then it must contain the entire $\ov {u_j}$. It follows that either either $w\equiv \ov u_i$ for some $i$, or $w\equiv \ov {u_{i-1}}a_{i}\cdots a_k\ov{u_k}$ for some $1\leq i\leq k\leq m$.  But in the latter case $u_{i-1}a_i\cdots a_ku_k$ is an invertible word, a contradiction.  Thus each $\ov {u_j}$ is a maximal invertible subword of $v$. It now follows that $v$ is reduced since the left-hand side of each rewrite rule is nonempty and invertible, and each maximal invertible subword of $v$ is reduced. Finally,  by Lemma~\ref{lem:Zhang:Irred}, each $\ov{u_i}$ is a word in $\Delta^*$. 
\end{proof}
Let us say that a word $u \in A^*$, written in Otto-Zhang normal form
\[u\equiv u_0a_1u_1 \ldots a_m u_m \]
has \emph{no invertible suffix} if $u_m$ represents the identity of $M$, that is $\ov{u_m}$ is empty.  We define having no invertible prefix dually. In particular, a nonempty \emph{reduced} word $\mu$ has \emph{no invertible suffix} if the Otto-Zhang normal form of $\mu$ is $\mu\equiv \mu_0a_1\mu_1\cdots \mu_na_n$ with the $a_i\in A$ and the $\mu_i$ the maximal invertible subwords of $\mu$.  In particular, $\mu$ is not invertible. There is an obvious dual condition for any nonempty reduced word to have no invertible prefix.

\section{Proving that $N \cong G \ast C^*$ for some finite set $C$.}\label{sec:FreeProd:special}

Throughout this section
\[
M = \lb A \mid w_1=1, \ldots, w_k=1 \rb.
\]

Let $R_1$ and $L_1$ be the right and left units, respectively, of the monoid $M$.
Let $G = R_1 \cap L_1$ be the group of units of $M$.
Define a monoid
\[
N = \{ (u,v) : u \in R_1, v \in L_1, uv=1 \}
\]
with multiplication
\[
(a,b) \cdot (x,y) = (ax,yb).
\]
Note that $N$ is a submonoid of $R_1\times L_1^{\mathrm{op}}\subseteq M \times M^{\mathrm{op}}$.

We remark that while $N$ is a set of pairs of elements of $M$ we shall almost always write elements of $N$ as pairs $(u,v)$ where $u, v \in A^*$. This should of course be interpreted as meaning the pair $([u],[v])$ where $[u], [v] \in M$ are the elements of $M$ represented by the words $u$ and $v$, respectively.

By results of Zhang~\cite[Theorem 4.4]{Zhang} we know that $R_1 \cong G \ast C^*$ for some finite set $C$, and $L_1 \cong G \ast D^*$ for some finite set $D$, it follows that $N$ is a submonoid of $(G \ast C^*) \times (G \ast D^*) ^{\mathrm{op}}$. This implies that $N$ is group-embeddable, and thus, in particular, $N$ is two-sided cancellative. More generally $R_1 \times L_1^{op}$ is a group-embeddable and thus two-sided cancellative monoid.

Our aim in this section is to prove that $N \cong G \ast X^*$ for some finite set $X$.
The key to proving this result is to identify a suitable generating set with respect to which we can see that $N$ has this free product structure.

Given any invertible word $\alpha\in A^*$ we shall use $\alpha^{-1}$ to denote the reduced form of the inverse of the invertible element $\alpha$.
For example if $M = \lb A \mid r=1 \rb$ where $r \equiv (ab)(cd)(ab)$ then $(cd)(ab) = (ab)(cd)$ is the inverse of $ab$ and it may be verified that $(ab)^{-1} = (ab)(cd)$ is the normal form in this case.
By Lemma~\ref{lem:Zhang:Irred} we know that $\alpha^{-1} \in \Delta^*$ for every invertible word $\alpha\in A^*$, and if $\alpha$ is itself irreducible, then also $\alpha\in \Delta^*$.

Let $u,v\in A^+$ be reduced words such that $uv=1$ in $M$.  Then by definition of the Zhang rewriting system and the fact that $u,v$ are reduced, there must exist factorisations $u\equiv u_1u_2$ and $v\equiv v_1v_2$ with $u_2v_1\in \Delta^*$ and $u_2$ and $v_1$ both nonempty.  Of course, $u_2v_1$ is an invertible word.  In general, if $u,v$ are nonempty words, we say that an invertible word $z\in A^+$ \emph{cuts across} $(u,v)\in A^+\times A^+$ if $u\equiv u_1u_2$ and $v\equiv v_1v_2$ with $z\equiv u_2v_1$ and $u_2,u_1$ nonempty.
Let $(u,v) \in N$ with $u$ and $v$ both reduced. It follows from the definitions that $u=1$ in $M$ if and only if $v=1$ in $M$. Otherwise, $u \in A^+$ and $v \in A^+$ and the argument above shows that at least one invertible word cuts across $(u,v)$.
The collection of invertible words cutting across $(u,v)$ is partially ordered by subword inclusion, as subwords of $uv$.  

Here a subword of $uv$ cutting across $(u,v)$ is specified by choosing a letter $x$ in $u$ and a letter $y$ in $v$ and taking the unique subword of $uv$ that begins with that particular letter $x$ and ends with that particular letter $y$. Distinct pairs of such letters $x,y$ determine distinct subwords of $uv$, viewed as subwords of $uv$, even if graphically they are the same word. For example, suppose $M = \lb a,b \mid (ab)^6=1 \rb$. Let $u \equiv aba$ and $v \equiv bab$ so that $uv \equiv ababab$. Then this word $uv$ contains the word $abab$ as a subword in two different ways, once as a prefix $\underline{abab}ab$ of $uv$ and once as a suffix $ab\underline{abab}$ of $uv$. Each of these subwords $abab$ of $uv$ cuts across $(u,v) = (aba,bab)$ but, even though they are both equal to the word $abab$, they are distinct as subwords of $uv$, and thus they are two distinct elements in the poset of invertible words cutting across $(u,v)$.  
Furthermore since as subwords of $uv$ neither of these words $abab$ is contained in the other, they are incomparable in this poset. Note that the subword $ab$ cutting across $uv$, which is the underlined word in $ab\underline{ab}ab$, is also an element of the poset of invertible words cutting across $(u,v)$, and this element of the poset is contained, as a subword of $uv$, in both the subwords $abab$ of $uv$ and hence lies below each of those elements of the poset. 

Returning to the general argument, 
we now aim to show that set of invertible words cutting across $(u,v)$, if nonempty, is a lattice.  Let us prove a more general result first.

\begin{Prop}\label{p:lattice.vers1}
Let $u\in A^+$ and $v$ be a nonempty subword of $u$. If some invertible subword of $u$ contains $v$, then the set of invertible subwords of $u$ containing $v$ is a lattice ordered by subword inclusion.  Moreover, if $\delta\equiv u_1vu_2$ is the unique smallest element of this lattice, then $u_1$ has no nonempty invertible prefix and $u_2$ has no nonempty invertible suffix.
\end{Prop}
\begin{proof}
We being by showing the poset is a lattice. Let $x,y\in A^+$ be invertible subwords of $u$ containing $v$.  We must show that $x,y$ have a join and meet in the collection of invertible subwords of $v$ containing $u$.  Write $x\equiv x_1vx_2$ and $y\equiv y_1vy_2$.  Without loss of generality assume that $|x_1|\leq |y_1|$.  If $|x_2|\leq |y_2|$, then $x$ is a subword of $y$ and so the meet and join in this case are clearly given by $x$ and $y$, respectively.  So assume that $|x_2|>|y_2|$.  Then it suffices to show that $y_1vx_2$ and $x_1vy_2$ are invertible, as these will then be the meet and join,  respectively, of $x,y$ in this poset.  Write $y_1\equiv \alpha x_1$ and $x_2\equiv y_2\beta$.  Then $y_1vy_2\equiv \alpha x_1vy_2$ and $x_1vx_2\equiv x_1vy_2\beta$ are invertible, and so $x_1vy_2$ and $y_1vx_2\equiv \alpha x_1vy_2\beta$ are invertible by Lemma~\ref{lem:unions:of:invertible:general}, as required.

For the final statement, if $u_1\equiv u_1'u_1''$ with $u_1'$ invertible, then $u_1''vu_2$ is invertible and hence by minimality of $u$, we must have $u_1''\equiv u_1$, and so $u_1'$ is empty.  Dually $u_2$ has no nonempty invertible suffix.
\end{proof}

\begin{Prop}\label{p:lattice.cuts}
Let $(u,v)\in A^+\times A^+$.  Suppose that some invertible word cuts across $(u,v)$.  Then the poset of invertible words cutting across $(u,v)$ is a lattice.
\end{Prop}
\begin{proof}
Let $a$ be the last letter of $u$ and $b$ the first letter of $v$ and consider the subword $ab$ of the word $uv$ where $ab$ cuts across $(u,v)$. Then the invertible words cutting across $(u,v)$ are the subwords of $uv$ containing this particular subword $ab$ of $uv$.  Therefore, this poset is a lattice by Proposition~\ref{p:lattice.vers1}.
\end{proof}

\begin{definition}\label{def:height}
Given any pair $(\mu,\gamma)\in A^+\times A^+$ such that some invertible word cuts across $(\mu,\gamma)$, we define the height of $(\mu,\gamma)$ to be the height of the lattice of invertible words cutting across $(\mu,\gamma)$, i.e., the length of the longest chain of invertible words cutting across $(\mu,\gamma)$.
  \end{definition}

Clearly since the lattice in Definition~\ref{def:height} is always finite, it must have finite height, and so the height of a pair of words $(\mu,\gamma)\in A^+\times A^+$, such that some invertible word cuts across $(\mu,\gamma)$, is always finite.   

We shall later need the following property of height $0$ pairs.

\begin{lemma}\label{lem:ht.zero}
Let $\mu,\gamma\in A^*$ be reduced  with $\mu\gamma$ invertible but $\mu$ noninvertible (or, equivalently, $\gamma$ noninvertible).  If $(\mu,\gamma)$ has height $0$, then $\mu$ has no nonempty left invertible prefix and $\gamma$ has no nonempty right invertible suffix.
\end{lemma}
\begin{proof}
Write $\mu\equiv \mu'a$ and $\gamma\equiv b\gamma'$ with $a,b\in A$.   Note that $\mu'$ is right invertible and $\gamma'$ is left invertible, and so  $\mu'$ has no nonempty left invertible prefix and $\gamma'$ has no nonempty right invertible suffix by Proposition~\ref{p:lattice.vers1}.  Since $\mu$ and $\gamma$ are also not invertible, but are right and left invertible, respectively, the lemma follows.
\end{proof}

We are now ready to provide a finite generating set for $N$.

\begin{lemma}
\label{lem_NGenAdvanced_newExtra}
The monoid $N$ is generated by the set
$X \cup Y$ where
\[Y=\{(\delta,\delta^{-1}): \delta\in \Delta\}\] and
$X$ consists of all pairs $(\delta_1,\delta_2\delta\inv)$ such that:
\begin{enumerate}
\item  $\delta_1,\delta_2\in A^+$ are reduced;
\item  $\delta\equiv \delta_1\delta_2$ is invertible;
\item $(\delta_1,\delta_2)$ has height $0$;
\item  $\delta_1$ has no invertible suffix.
\end{enumerate}
Furthermore these four conditions imply that $\delta \equiv \delta_1 \delta_2 \in \Delta$, and hence the set $X$ is finite.  
Therefore $N$ is finitely generated by the set $X \cup Y$.  
\end{lemma}
\begin{proof}
Let $u,v\in A^*$ be reduced with $uv=1$ in $M$.  We prove by induction on $|u|$ that $(u,v)\in \langle X\cup Y\rangle$. If $u=1$, then since $uv=1$ in $M$ and $v$ is reduced, we must have $v=1$. Proceeding by induction, if $|u|>0$,  then $uv$ is not reduced (being $1$ in $M$), and so we can apply a rewrite rule to it.  Since $u,v$ are reduced, the left-hand side of any applicable rewrite rule cuts across $(u,v)$.  But the left-hand side of any rewrite rule is invertible and nonempty.  Hence there is an invertible word cutting across $(u,v)$.  Let  $\delta$ be the minimal invertible word cutting across $(u,v)$; this exists by Proposition~\ref{p:lattice.cuts}.  Write $u\equiv u'\delta_1$ and $v\equiv \delta_2v'$ with $\delta\equiv \delta_1\delta_2$. Suppose first that $\delta_1$ has no invertible suffix.  We compute $(u,v)=(u',\delta v')(\delta_1,\delta_2\delta\inv)$.  Note that $|u'|<|u|$ and $u'\delta v'\equiv uv$, and hence is $1$ in $M$.  So $(u',\delta v')\in \langle X\cup Y\rangle$ by induction, noting that $u'$ is a reduced word, whereas $(\delta_1,\delta_2\delta\inv)\in X$ by minimality of $\delta$ which implies $(\delta_1, \delta_2)$ has height zero.  Otherwise, $\delta_1\equiv \delta_1'\gamma$ where $\gamma$ is a nonempty invertible suffix.  Since $\delta_1$ is reduced, $\gamma\in \Delta^*$ by Lemma~\ref{lem:Zhang:Irred}.  Therefore, $(\gamma,\gamma\inv)\in \langle Y\rangle$.  But then $(u,v) = (u'\delta_1'\gamma,v) = (u'\delta_1',\ov{\gamma v})(\gamma,\gamma\inv)$, where $u' \delta_1'$ is a reduced word since it is a subword of the reduced word $u$. Since $u'\delta_1'\ov{\gamma v}=uv=1$ in $M$, and $|u'\delta_1'|<|u|$, we deduce that  $(u'\delta_1'\gamma,v)\in \langle X\cup Y\rangle$ by induction, and hence $(u,v)\in \langle X\cup Y\rangle$.

Finally, we show that conditions (1)--(4) together imply that $\delta\equiv\delta_1\delta_2 \in \Delta$ and hence the set $X$ is finite. There are two cases to consider. 

First suppose that $\delta\equiv\delta_1\delta_2$ is not reduced. Since $\delta_1$ and $\delta_2$ are both reduced by (1) it follows that the left hand side $\mu \in \Delta^+$ of some rewrite rule cuts across $(\delta_1, \delta_2)$.          
Write $\delta_1\delta_2 \equiv \delta_1' \mu_1 \mu_2 \delta_2'$ where $\delta_1 \equiv \delta_1' \mu_1$, $\delta_2 \equiv \mu_2 \delta_2'$ and $\mu_1\mu_2 \in \Delta^+$ with $\mu_1$ and $\mu_2$ both nonempty words. Since by (4) the word $\delta_1$ has no invertible suffix, it follows that $\mu_1 \not\in \Delta^+$ hence there is a word $\gamma \in \Delta$ such that $\gamma$ cuts across $(\delta_1,\delta_2)$. Since by (3) we have that $(\delta_1,\delta_2)$ has height zero this is only possible if $\delta \equiv \delta_1 \delta_2 \equiv \gamma \in \Delta$. This completes the proof that in this case $\delta \equiv \delta_1\delta_2 \in \Delta$. 

Next we consider the case that $\delta\equiv\delta_1\delta_2$ is reduced. Since $\delta_1\delta_2$ is reduced, then as it is invertible it belongs to $\Delta^*$ by Lemma~\ref{lem:Zhang:Irred}.  Since $\delta_1\notin \Delta^*$ (being noninvertible), the expression of $\delta_1\delta_2$ as a product of words from $\Delta$ must involve an element $\delta'\in\Delta$ cutting across $(\delta_1,\delta_2)$.  Since any element of $\Delta$ is invertible, we deduce that $\delta_1\delta_2\equiv \delta' \in \Delta$ by definition of height $0$. This completes the proof that in this case $\delta \equiv \delta_1\delta_2 \in \Delta$. 

We have shown that the conditions (1)--(4) together imply that $\delta\equiv\delta_1\delta_2 \in \Delta$ and hence, since $\Delta$ is finite, the set $X$ is also finite. 
\end{proof}

\begin{Rmk}
The idea behind the height zero condition in the definition of the set $X$ in Lemma~\ref{lem_NGenAdvanced_newExtra} is to ensure that the generating set obtained in minimal. In particular, if we were to remove the height zero assumption from the set $X$ in Lemma~\ref{lem_NGenAdvanced_newExtra} then in general it would result in a generating set for $N$ that is not minimal. To see this consider, for example, the monoid $M$ defined by the presentation $\lb a,b,c,d \mid (abcd)(bc)(abcd) = 1 \rb$. It is straightforward to verify that the minimal invertible pieces of this defining relator are $abcd$ and $bc$.  
Consider the following three elements from $N$   
\[
(b,c(bc)^{-1}), \quad  
(ab,cd(abcd)^{-1}), \quad 
(a,bcd(abcd)^{-1}).  
\] 
All three of theses elements has the form $(\delta_1, \delta_2 \delta^{-1})$ where $\delta \equiv \delta_1\delta_2 \in \Delta$ since it may be shown that $bc \in \Delta$ and $abcd \in \Delta$. However, if we included all three of these generators in a generating set for $N$ then that generating set would not be minimal since
\[
(a,bcd(abcd)^{-1}) \cdot (b,c(bc)^{-1}) = 
(ab,c(bc)^{-1}bcd(abcd)^{-1})  = 
(ab,cd(abcd)^{-1}).   
\]
Note that the pair $(ab,cd(abcd)^{-1})$ does not have height zero, and the product above shows that this element can be written as a product of the two height zero elements  
$(a,bcd(abcd)^{-1})$ and $(b,c(bc)^{-1})$.
\end{Rmk}

We are now in a position to prove that $N$ decomposes as a free product of its group of units and a free monoid of finite rank.  Let us recall briefly the normal form for a free product.  Suppose that $\{M_i:i\in I\}$ is a collection of monoids.  Consider the alphabet $B=\bigsqcup_{i\in I}M_i\setminus \{1\}$.  An element of $M_i\setminus \{1\}$ is called an $M_i$-syllable or a syllable of type $M_i$.  A word in $B^*$ is in normal form if it does not contain two consecutive syllables of the same type.  The length of a normal form is called its syllable length.  Each element of the free product $\Asterisk_{i\in I} M_i$ has a unique representative in normal form.

The following is a variant of the Ping Pong Lemma from group theory (see e.g. 
\cite[Chapter II.B]{de2000topics} or \cite[Proposition 12.2 of Chapter III]{LyndonAndSchupp}).

\begin{lemma}[Ping Pong Lemma]\label{l:ping.pong}
Let $N$ be a cancellative monoid, $G$ a subgroup of the group of units of $N$ and $X\subseteq N$.  Suppose that $N=\langle G\cup X\rangle$ and there is a right $N$-set $A$  such that $A=A_0\sqcup \bigsqcup_{x\in X}Ax$ with $Axg\subseteq A_0$ whenever $x\in X$ and $1\neq g\in G$, and $Ax\subsetneq A$. Then $\langle X\rangle$ is a free monoid with basis $X$ and $N\cong G\ast X^*$.
\end{lemma}
\begin{proof}
First note that if $X=\emptyset$, then $N=G$ and there is nothing to prove.  So we may assume that $X\neq \emptyset$.
For each $x \in X$ if $x^i=x^j$ for some $i,j \in \mathbb{N}$ with $i > j$ then since $N$ is a cancellative monoid this would imply $x^{i-j}=1$ and thus $x$ is invertible. But if $x$ were invertible it would imply $Ax=A$ contradicting the hypotheses of the lemma. 
Hence, 
each $x\in X$ generates a free submonoid $x^*$. 
We want to show that $N \cong G\ast \Asterisk_{\substack{x\in X}}x^*$.  
By assumption, the natural homomorphism $\phi: G\ast \Asterisk_{\substack{x\in X}}x^* \rightarrow N$ from the free product to $N$ is surjective.  
The action of $N$ on $A$ can be used to define an action of $G\ast \Asterisk_{\substack{x\in X}}x^*$ on $A$ via the homomorphism $\phi$. 
More precisely, we define an action of $G\ast \Asterisk_{\substack{x\in X}}x^*$ on $A$ where for each $s \in G\ast \Asterisk_{\substack{x\in X}}x^*$ and $a \in A$ we define $a s = a \phi(s)$.  
Next note that if 
an element $\alpha$ of the free product $G\ast \Asterisk_{\substack{x\in X}}x^*$
is in normal form with last syllable of type $x^*$, then $A\alpha\subseteq Ax$. 
If $\alpha$ has last syllable of type $G$, there are two cases.  If $\alpha\in G\setminus \{1\}$, then $A\alpha=A$ and $Ax\alpha\subseteq A_0$ for all $x\in X$; if $\alpha=\alpha'x^kg$ with $x\in X$, $k>0$ and $g\in G\setminus \{1\}$, then $A\alpha\subseteq A_0$.  It follows that if $\alpha$ is a nonempty normal form, then it acts nontrivially on $A$, and we can determine the type of its last syllable from its action on $A$.   

 Suppose that $\alpha$ and $\beta$ are normal forms for the free product.  We induct on the sum of their syllable lengths to show that if $\alpha$ and $\beta$ represent the same element of $N$, then they are equal in the free product. By the previous paragraph, we may assume that $\alpha$ and $\beta$ are nonempty and  their last syllables have the same type.
Suppose first that  $\alpha = \alpha'x^k$ and $\beta=\beta'x^m$ with $x\in X$ and $\alpha',\beta'$ of one shorter syllable length.  Without loss of generality assume that $k\leq m$.  Then cancelling $x^k$ we get $\alpha'$ and $\beta'x^{m-k}$ represent the same element of $N$.  By induction, since $\alpha'$ does not end in an $x^*$-syllable, we must have that $m=k$ and  $\alpha'=\beta'$ in the free product.   Next suppose that $\alpha=\alpha' g$ and $\beta =\beta'g'$ where $g,g'\in G\setminus \{1\}$ and $\alpha',\beta'$ have syllable length one shorter.  Then $\alpha'=\beta' g'g\inv$ in $N$.  Since $\alpha'$ does not end in a $G$-syllable, we deduce by induction that $g'g\inv =1$ and $\alpha'=\beta'$ in the free product.  This completes the proof.
\end{proof}

We now apply the Ping Pong Lemma to our setting.

\begin{theorem}\label{thm:FreeProductStructure}
The monoid $N$ is isomorphic to a free product $G \ast C^*$ where $G$ is the group of units of $M$ and $C$ is a finite set. More precisely,
let
$X \cup Y$ be the finite generating set for $N$ obtained in Lemma~\ref{lem_NGenAdvanced_newExtra}. Then $X$ is a basis for a free submonoid of $N$ and $N\cong G\ast X^*$ where the copy of $G$ is identified with $\langle Y\rangle$.
  \end{theorem}
\begin{proof}
 Note that if $\mu,\gamma \in A^*$ are reduced with $\mu$ nonempty, 
 and with $\mu\gamma$ invertible and with $\mu$ having no invertible suffix, then $\mu$ is right invertible, but not invertible.

We prove the free product decomposition using the Ping Pong Lemma applied to the right action of $N$ on itself.  We know that $N$ is generated by $X\cup Y$ by Lemma~\ref{lem_NGenAdvanced_newExtra}.  We have already pointed out that $N$ is cancellative. Note that if $x=(\delta_1,\delta_2\delta\inv)\in X$ with $\delta_1\delta_2\equiv \delta\in \Delta$ and $\delta_1$ having no invertible suffix, then $\delta_1$ is right invertible (as $\delta$ is invertible), and hence is not left invertible.  It follows that $x$ is not left invertible, and so  $Nx\subsetneq N$ for all $x\in X$.  We now prove that the collection $\{Nx: x\in X\}$ is pairwise disjoint.

Let $\pi\in N$ with $\pi=(\pi_{\ell},\pi_r)$ where $\pi_{\ell}$ and $\pi_r$ are reduced words.  If $(\mu,\gamma\tau\inv)\in X$ with $\mu\gamma\equiv \tau\in \Delta$, we shall show that we can uniquely recover $(\mu,\gamma\tau\inv)$ from $(\pi_\ell\mu, \gamma\tau\inv \pi_r)=\pi(\mu,\gamma\tau\inv)$.
We need a claim in order to do this.  We continue to write $\ov{w}$ for the reduced form of $w$.

\begin{Claim}\label{claim:1}
Let $\pi=(\pi_\ell,\pi_r)\in N$ with $\pi_{\ell},\pi_r$ reduced, $(\mu,\gamma\tau\inv)\in X$ with 
$\mu\gamma\equiv \tau\in \Delta$
and $\delta\in \Delta^*$ reduced.  Then
\[\pi(\mu,\gamma\tau\inv)(\delta,\delta\inv) = (\pi_{\ell}\mu\delta,\delta\inv\gamma\ov{\tau\inv \pi_r})\] and $\pi_{\ell}\mu\delta$ and $\gamma\ov{\tau\inv \pi_r}$ are both reduced.  
\end{Claim}  
\begin{proof}[Proof of Claim~\ref{claim:1}]
By assumption $\pi_{\ell}$, $\pi_r$ and $\delta$ are reduced.  The words $\mu,\gamma$ are reduced by the definition of $X$ in Lemma~\ref{lem_NGenAdvanced_newExtra}.  First we observe that $\pi_{\ell}\mu$ is reduced.  This follows since $\pi_{\ell}$ and $\mu$ are reduced, and no rewrite rule can be applied with the left-hand side of the rewrite rule including both letters in $\pi_{\ell}$ and in  $\mu$ since the left-hand side of every rewrite rule belongs to $\Delta^+$, and hence is invertible,  and that would then give a nonempty left invertible prefix of $\mu$, something we already observed in Lemma~\ref{lem:ht.zero} it does not have.  Dually, $\gamma\ov{\tau\inv\pi_r}$ is reduced because no rewrite rule can be applied to $\gamma$ or $\overline{\tau^{-1}\pi_r}$ since they are reduced, and
furthermore, any rewrite rule applied across both of these words (i.e., involving letters from both of them) would contradict the fact (proved in Lemma~\ref{lem:ht.zero}) that $\gamma$ has no nonempty right invertible suffix, as the left-hand side of any rewrite rule is nonempty and invertible.  

Next we prove that $\delta$ is a maximal invertible subword of $\pi_{\ell}\mu\delta$. 
Indeed, by definition of the generating set $X$ the word $\mu$ does not have an invertible suffix.
It follows that there is no suffix $\mu''$ of $\mu$ such that $\mu''\delta$ is invertible.
Also, a suffix of $\pi_{\ell}\mu\delta$ of the form $\sigma \mu\delta$ cannot be invertible since if it were, then $\sigma \mu$ would be invertible, implying  $\mu$ is left invertible. But since by definition $\mu$ is right invertible this would imply $\mu$ is invertible, and this would contradict the fact that $\mu$ has no invertible suffix. 
Hence $\delta$ is a maximal invertible subword of $\pi_{\ell}\mu\delta$. 

Finally we show that $\pi_{\ell}\mu\delta$ is reduced. Since $\pi_{\ell}\mu$ and $\delta$ are both reduced the left hand side of any rewrite rule applied to $\pi_{\ell}\mu\delta$ must cut across $(\pi_{\ell}\mu,\delta)$. But 
the left hand side of any rewrite rule is nonempty and invertible so there would be an invertible word cutting across $(\pi_{\ell}\mu,\delta)$. But then since $\delta$ is invertible, by Lemma~\ref{lem:unions:of:invertible:general} that would imply that $\delta$ is not a maximal invertible subword of $\pi_{\ell}\mu\delta$, contradicting what was proved in the previous paragraph. Hence $\pi_{\ell}\mu\delta$ is reduced.
This completes the proof of the claim.
  \end{proof}
From Claim~\ref{claim:1} it follows that $\pi_{\ell}\mu$ and $\gamma\ov{\tau\inv\pi_r}$ are reduced.
We now recover $(\mu,\gamma\tau\inv)$ uniquely from $(\pi_\ell\mu,\gamma\ov{\tau\inv \pi_r})$ in the following way.
Consider all the invertible words cutting across this pair.
There is at least one such word since $\mu\gamma$ is invertible and $\mu$ and $\gamma$ are both nonempty words by definition of $X$.
By Proposition~\ref{p:lattice.cuts} the set of invertible words cutting across this pair is a lattice
under inclusion. 
We claim that the word
$\mu \gamma$ is the unique minimal element of this lattice.  This word is minimal because, by definition of $X$, the pair $(\mu, \gamma)$ has height zero.  In a lattice minimal elements are unique.  It follows that we can recover $(\mu, \gamma\tau\inv)$ from $(\pi_\ell\mu,\gamma\ov{\tau\inv \pi_r})$, and hence the collection $\{Nx: x\in X\}$ is pairwise disjoint.

Let $N_0$ be the complement of $\bigsqcup_{x\in X}Nx$.  By the Ping Pong Lemma, it remains to show that $Nx(g,g\inv)\subseteq N_0$ for all $x\in X$ and $g\neq 1$ in $G$.  Indeed, let $x=(\mu,\gamma\tau\inv)\in X$, and let $\delta\in \Delta^*$ be reduced representing $g$, whence $\delta$ is nonempty.  Let $\pi=(\pi_\ell,\pi_r)\in N$ with $\pi_\ell,\pi_r$ reduced.  Then 
$\pi x(g,g\inv) = (\pi_\ell\mu\delta,\delta\inv \gamma\tau\inv \pi_r)$ and $\pi_\ell\mu\delta$ is reduced by Claim~\ref{claim:1}.  Now if $\pi x(g,g\inv)=\pi'x'$ with $\pi'=(\pi_\ell',\pi_r')\in N$, $x'=(\mu',\gamma'(\tau')\inv)\in X$, $\tau'\equiv \mu'\gamma'\in \Delta$ and $\pi_\ell',\pi_r'$ reduced, then we would have $\pi'x'=(\pi_\ell'\mu',\gamma'(\tau')\inv\pi_r')$ with $\pi_\ell'\mu'$ reduced by Claim~\ref{claim:1}.  Thus $\pi_\ell'\mu'\equiv\pi_\ell\mu\delta$.  We cannot have $\delta$ as a suffix of $\mu'$ since $\mu'$ has no invertible suffix and $\delta$ is nonempty.   If $\mu'$ is a suffix of $\delta$, then it is left invertible, but it is also right invertible since $\mu'\gamma'$ is invertible.  This contradicts that $\mu'$ has no invertible suffix. Thus $\pi x(g,g\inv)\in N_0$.  This completes the proof that $X$ freely generates $X^*$ and  $N\cong G\ast X^*$ by the Ping Pong Lemma. 
\end{proof}

\section{Proving $M \times M$ is a free $N$-set}\label{sec:MxMFreeNSet:special}

As in the previous section, throughout this section
\[
M = \lb A \mid w_1=1, \ldots, w_k=1 \rb.
\]

For the next part of the argument we want to prove that $M \times M$ is a free right $N$-set under the action
\[
(m_1,m_2) \cdot (n_1,n_2) = (m_1 n_1, n_2 m_2),
\]
where $(m_1,m_2) \in M \times M$ and $(n_1,n_2) \in N$. Note that in the above situation we have
\[
m_1m_2 = m_1 n_1 n_2 m_2
\]
in $M$ since $n_1n_2=1$ by definition of $N$. 

We shall need the following remark. If $S$ is a cancellative monoid and $X$ is a free right $S$-set, then $xs=xs'$ implies $s=s'$ for $x\in X$ and $s,s'\in S$, and dually for free left $S$-sets.  Indeed, this is true for the right action of $S$ on itself by cancellativity, and a free $S$-set is a disjoint union of copies of the action of $S$ on itself.

Our aim in this section is to prove the following result. Note the basis in the following theorem is not finite, but we do not expect nor need it to be so.

\begin{theorem}\label{thm:Free:N:Set2}
The set $M \times M$ is a free right $N$-set under the action
\[
(m_1,m_2) \cdot (n_1,n_2) = (m_1 n_1, n_2 m_2),
\]
where $(m_1,m_2) \in M \times M$ and $(n_1,n_2) \in N$. In more detail, let
$\mathcal{B}$ be the set of all pairs of words $(\alpha_1, \alpha_2)$ such that 
$\alpha_1\alpha_2 \equiv \alpha$ is a reduced word and
$\alpha$ has Otto-Zhang normal form \[\alpha\equiv w_0a_1w_1\cdots a_mw_m\]
with $\alpha_1\equiv w_0a_1\cdots wa_{i-1}$  and $\alpha_2\equiv w_{i-1}a_i\cdots a_mw_m$
for some $1\leq i\leq m+1$.
Then $M \times M$ is a free right $N$-set with basis
$\mathcal{B}$.
 \end{theorem}
\begin{proof}
Let $u,v\in A^*$ be reduced words and let $uv\equiv w_0a_1\cdots a_mw_m$ in Otto-Zhang normal form.  Suppose, moreover, that $u\equiv w_0a_0\cdots w_{i-2}a_{i-1}w_{i-1}'$ and $v\equiv w_{i-1}''a_i\cdots a_mw_m$ where $w_{i-1}\equiv w_{i-1}'w_{i-1}''$.   Note that for $j\neq i-1$, we have that $w_j$ is reduced and hence belongs to $\Delta^*$ by Lemma~\ref{lem:Zhang:Irred}.  Also, $\ov{uv} = w_0a_1\cdots a_{i-1}\ov{w_{i-1}}a_i\cdots a_m w_m$ by Lemma~\ref{lemma:OZ}.   It follows that $(u',v')\in \mathcal B$ where $u'\equiv w_0a_1\cdots w_{i-2}a_{i-1}$ and $v'\equiv \ov{w_{i-1}}a_i\cdots a_mw_m$.  We claim that $(u,v)\in (u',v')N$.  Indeed, note that $(w_{i-1}',w_{i-1}''w_{i-1}\inv)\in N$ 
since 
$w_{i-1}' w_{i-1}''w_{i-1}\inv \equiv  w_{i-1}w_{i-1}\inv = 1$,
and
\[(u',v')(w_{i-1}',w_{i-1}''w_{i-1}\inv) = (u'w_{i-1}', w_{i-1}''w_{i-1}\inv w_{i-1}a_i\cdots a_mw_m) = (u,v),\]
as required.  We conclude that $\mathcal B$ is a generating set for $M\times M$ as a right $N$-set.  It remains, to show each element of $M\times M$ has only one expression in the form $(\alpha_1,\alpha_2)(n_1,n_2)$ with $(\alpha_1,\alpha_2)\in \mathcal B$ and $(n_1,n_2)\in N$.

Suppose that $(\alpha_1,\alpha_2)\in \mathcal B$ and let $x,y$ be reduced words with $xy=1$.  We show that $(\alpha_1,\alpha_2)$ can be recovered from the right-hand side of $(\alpha_1,\alpha_2)(x,y) = (\alpha_1x,y\alpha_2)$.  Note that by definition of $\mathcal B$, $\alpha_1$ has no invertible suffix.   

First we claim that $\alpha_1x$ is reduced.  Indeed, since $\alpha_1$ and $x$ are reduced, any rewrite rule that is applicable must cut across $(\alpha_1,x)$.  Since the left-hand side of a rewrite rule is nonempty and invertible, this means we can write $\alpha_1\equiv \alpha\beta$ and $x\equiv \gamma\delta$ with $\beta\gamma$ invertible and $\beta$, $\gamma$ nonempty.  Since $x$ is right invertible, we deduce that $\gamma$ is invertible and hence $\beta$ is invertible, contradicting that $\alpha_1$ has no invertible suffix.
Hence $\alpha_1x$ is reduced.

Consider $(\alpha_1x,\ov{y\alpha_2})$. 
Let $\mu$ be the longest invertible word that cuts across $(\alpha_1x,\ov{y\alpha_2})$ if such a word exists, while if no such word exists then let $\mu$ be the longest invertible suffix of $\alpha_1 x$, which could possibly be empty. We claim that $\alpha_1x\ov{y\alpha_2}\equiv \alpha_1\mu\nu$ with $\ov{\mu\nu}=\alpha_2$.  It will then follow that we can recover $(\alpha_1,\alpha_2)$ from $(\alpha_1,\alpha_2)(x,y)$.    

\begin{sloppypar}
Indeed, to prove the claim let $\alpha_1\alpha_2\equiv w_0a_0\cdots a_mw_m$ in Otto-Zhang normal form. 
Each $w_j$ is reduced because $\alpha_1\alpha_2$ is reduced, by the definition of $\mathcal B$, and then $w_j \in \Delta^*$ by Lemma~\ref{lemma:OZ}.
Then by definition of $\mathcal B$, we have some $i$ with $\alpha_1\equiv w_0a_0\cdots w_{i-2}a_{i-1}$ and $\alpha_2\equiv w_{i-1}a_i\cdots a_mw_m$.   Then $\alpha_1xy\alpha_2\equiv w_0a_0\cdots w_{i-2}a_{i-1}xyw_{i-1}a_i\cdots a_mw_m$.  By Lemma~\ref{lem:unions:of:invertible:general}, it follows that any maximal invertible subword of $\alpha_1xy\alpha_2$ is either a maximal invertible subword of $\alpha_1$, $\alpha_2$ or contains $xy$.  But since $xy=1$ in $M$ and $w_{i-1}$ is maximal invertible in $\alpha_1\alpha_2$, it follows that $xyw_{i-1}$ is the maximal invertible subword of $\alpha_1xy\alpha_2\equiv w_0a_0\cdots w_{i-2}a_{i-1}xyw_{i-1}a_i\cdots a_mw_m$ containing $xy$. 
It follows that the Otto Zhang normal form of the word $\alpha_1xy\alpha_2$ is
\[
\alpha_1xy\alpha_2\equiv w_0a_0\cdots w_{i-2}a_{i-1}(xyw_{i-1})a_i\cdots a_mw_m
\]
where the maximal invertible subwords in this normal form are the words
\[
w_0, \ldots, w_{i-2}, xyw_{i-1}, w_i, \ldots, w_m. 
\]
Since $xyw_{i-1}$ is invertible it follows that $y w_{i-1}$ is left invertible. Since $a_i w_i \cdots a_mw_m$ has no invertible prefix it follows that no invertible word cuts across the pair
\[
(y w_{i-1}, a_i w_i \cdots a_mw_m).
\]
Indeed, if $\sigma \equiv \sigma_1 \sigma_2$ were an invertible word that cuts across this pair with $\sigma_2$ a prefix of   
$a_i w_i \cdots a_mw_m$, then $\sigma_1$ is left invertible since it is a suffix of the left invertible word $y w_{i-1}$, and $\sigma_1$ is right invertible since it is a prefix of the invertible word $\sigma$. Then $\sigma_1$ and $\sigma$ are invertible and thus $\sigma_2$ is an invertible prefix of $a_i w_i \cdots a_mw_m$ which is a contradiction. It follows that every invertible subword of the word       
\[
y w_{i-1} a_i w_i \cdots a_mw_m
\]
is contained in either a subword of $y w_{i-1}$ or else is a subword of $a_i w_i \cdots a_mw_m$.
It follows that the Otto-Zhang normal form of $y\alpha_2$ is graphically equal to the product of the Otto-Zhang normal form of $yw_{i-1}$ and $\nu:=a_iw_i\cdots a_mw_m$.  We conclude from Lemma~\ref{lemma:OZ} that $\ov{y\alpha_2} \equiv \ov{yw_{i-1}}a_iw_i\cdots a_mw_m$.  Then
\[\alpha_1 x \ov{y\alpha_2} \equiv w_0a_0\cdots w_{i-2}a_{i-1}x\ov{yw_{i-1}}a_iw_i\cdots a_mw_m.\]
Now there are two cases.
If there is an invertible word that cuts across $(\alpha_1x,\ov{y\alpha_2})$ then it follows that $\mu = x\ov{yw_{i-1}}$ since $xy=1$, $x\ov{yw_{i-1}}$ is invertible, and $w_{i-1}$ is a maximal invertible subword of $\alpha_1\alpha_2$.  Therefore, $\alpha_1\ov{y\alpha_2}\equiv \alpha_1\mu\nu$ and $\ov{\mu\nu}=\alpha_2$.
In the other case there is no invertible word that cuts across $(\alpha_1x,\ov{y\alpha_2})$.
In this case since $x \ov{yw_{i-1}}$ is invertible, either $x$ is empty or $\ov{yw_{i-1}}$ is empty. If $x$ is empty then $y$ is also empty since $xy=1$ and $x$ and $y$ are reduced. Therefore, $(\alpha_1 x, \ov{y \alpha_2}) = (\alpha_1, \alpha_2)$, and by definition of $\mathcal{B}$ the word $\alpha_1$ has no non-trivial invertible suffix, so $\mu$ is empty. So we recover $\alpha_2 = \nu = \ov{\mu \nu}$.  
Otherwise, $\ov{yw_{i-1}}$ is the empty word so
\[\alpha_1 x \ov{y\alpha_2} \equiv w_0a_0\cdots w_{i-2}a_{i-1}xa_iw_i\cdots a_mw_m.\] 
In this case in $M$ we have $x = xy w_{i-1} = w_{i-1}$. In particular $x$ is invertible. Furthermore
\[\alpha_1 x \equiv w_0a_0\cdots w_{i-2}a_{i-1}x \]
Suppose that $x$ is not the longest invertible suffix of the word $\alpha_1 x$; say, e.g., $z$ is a longer such suffix.  Then since $xy w_{i-1}$ and $z$ are both invertible with $x$ a suffix of $z$, it would follow 
from Lemma~\ref{lem:unions:of:invertible:general} 
that $zyw_{i-1}$ is an invertible word strictly containing $xy w_{i-1}$ contradicting the fact proved above that $xyw_{i-1}$ is a maximal invertible subword of $\alpha_1xy\alpha_2$. Hence in this case $x$ is the longest invertible suffix of the word $\alpha_1 x$. So by definition $\mu \equiv x$ and
\[\alpha_1 x \ov{y\alpha_2} \equiv \alpha_1 xa_iw_i\cdots a_mw_m \equiv \alpha_1 \mu \nu \] 
and
\[ \ov{\mu \nu} = \ov{xa_iw_i\cdots a_mw_m } = \ov{w_{i-1}a_iw_i\cdots a_mw_m } = \alpha_2 \]
since in this case, as we saw above, $x = w_{i-1}$ in $M$. This completes the proof that we can recover $(\alpha_1, \alpha_2)$ from $(\alpha_1, \alpha_2)(x,y)$.
\end{sloppypar}

Next suppose that $(\alpha_1,\alpha_2)\in \mathcal B$ and $(\alpha_1,\alpha_2)(x,y)=(\alpha_1,\alpha_2)(x',y')$ with $(x,y)$ and $(x',y')\in N$.  We show that $(x,y)=(x',y')$.  Indeed, by Corollary~\cite[Corollary~3.8]{GraySteinbergDocumenta}, $M$ is a free right $R_1$-set and a free left $L_1$-set. Note that $R_1,L_1$ are cancellative monoids. Hence the equations $\alpha_1x=\alpha_1x'$ and $y\alpha_2=y'\alpha_2$ imply that $x=x'$ and $y=y'$, as required. \end{proof}

\begin{remark}\label{remark:useful.basis}
Let $u,v \in A^*$ be reduced. Write
\begin{eqnarray*}
uv &\equiv& w_0a_1w_1 \ldots a_m w_m
  \end{eqnarray*}
as in the statement of Lemma~\ref{lem:OttoZhang}, where $a_i \in A$ for all $1 \leq i \leq m$ and $w_i$ is 
a maximal invertible subword of $w$ 
for all $0 \leq i \leq m$
(where some of the $w_i$ could be the empty word).  Suppose that
\begin{align*}
u & \equiv  w_0a_1\cdots w_{i-2}a_{i-1}w_{i-1}' \\ v & \equiv  w_{i-1}''a_iw_i\cdots a_mw_m
\end{align*}
with $w_{i-1}\equiv w_{i-1}'w_{i-1}''$.  The proof of Theorem~\ref{thm:Free:N:Set2} shows that $(u,v)$ is in the weak $N$-orbit of the basis element $(w_0a_1\cdots w_{i-2}a_{i-1},\ov{w_{i-1}}a_i\cdots a_mw_m)$. 

Since the set $M \times M$ is a free right $N$-set Theorem~\ref{thm:Free:N:Set2}, the observation in the previous paragraph provides a characterisation of exactly when two pairs $(u_1,v_1)$ and $(u_2,v_2)$ belong to the same weak $N$-orbit of the free right $N$-set $M \times M$. Indeed, $(u_1,v_1)$ and $(u_2,v_2)$ belong to the same weak $N$-orbit if and only if the basis elements from $\mathcal{B}$ obtained from $(u_1,v_1)$ and from $(u_2,v_2)$ via the process described in the previous paragraph are the same. This description of the weak $N$-orbits of $M \times M$ will be used in several arguments below. 
\end{remark}

Recall from the preliminaries section above that 
if a monoid $N$ acts on the right of a set $A$, then the weak orbits of $N$ on $A$ are the equivalence classes under the smallest equivalence relation on $A$ such that $a,an$ are equivalent for all $a\in A$ and $n\in N$.  We recall for the benefit of the reader that $\mathbb ZA\otimes_{\mathbb ZN} \mathbb Z$ is a free abelian group with basis the set of weak orbits of $N$ on $A$ via the map sending $a\otimes 1$ to the weak orbit of $a$. 

An important consequence of Theorem~\ref{thm:Free:N:Set2} and Remark~\ref{remark:useful.basis},
which will be proved in Lemma~\ref{lem:forst:quotient2} below, 
is that $\overleftrightarrow{\Gamma} /N$ is a forest, where $\overleftrightarrow{\Gamma}$ is the two-sided Cayley graph of $M$ and $\overleftrightarrow{\Gamma} /N$ is the graph obtained by collapsing the induced subgraph on each weak orbit of $N$ on the vertices of $\overleftrightarrow{\Gamma}$ to a point. So $\overleftrightarrow{\Gamma} /N$ has vertex set $(M\times M)/N$ and edge set consisting of those edges $e$ of $\overleftrightarrow{\Gamma}$ with initial vertex $\iota(e)$ in a different weak $N$-orbit than terminal vertex $\tau(e)$.
The corresponding edge of $\overleftrightarrow{\Gamma} /N$ goes from $[\iota(e)]$ to $[\tau(e)]$ where $[(m,m')]$ denotes the weak $N$-orbit of $(m,m')\in M\times M$.   
This gives a bijection between the edges of $\overleftrightarrow{\Gamma}/N$ and the set of edges of $\overleftrightarrow{\Gamma}$ whose initial and terminal vertices belong to distinct $N$-orbits. Thus for each such edge in $\overleftrightarrow{\Gamma}$ we can sensibly talk about the corresponding edge in $\overleftrightarrow{\Gamma}/N$. 
There is an $M\times M^{op}$-equivariant cellular map $\overleftrightarrow{\Gamma}\to \overleftrightarrow{\Gamma}/N$ sending a vertex to its class, and sending any edge connecting two vertices in the same weak $N$-orbit to that weak $N$-orbit (which is a vertex of the graph $\overleftrightarrow{\Gamma}/N$) and sending every other edge in $\overleftrightarrow{\Gamma}$ to the corresponding edge in $\overleftrightarrow{\Gamma}/N$.

Recall that $\pi_0(\overleftrightarrow{\Gamma})$ denotes the
 set of path-connected components of $\overleftrightarrow{\Gamma}$ which is isomorphic to $M$ as an $M \times M^{\mathrm{op}}$-set; see~\cite[Section~5.3]{GraySteinbergAGT}.
  Note that if $[(m_1,m_2)]=[(m_1',m_2')]$, then $m_1m_2=m_1'm_2'$.   It follows that the map $\overleftrightarrow{\Gamma}\to \overleftrightarrow{\Gamma}/N$ factors through the map $\overleftrightarrow{\Gamma}\to\pi_0(\overleftrightarrow{\Gamma})\cong M$, and so $\pi_0(\overleftrightarrow{\Gamma}/N)$ is identified with $M$.

\begin{lemma}\label{lem:forst:quotient2}
The quotient $\overleftrightarrow{\Gamma} /N$ of the two-sided Cayley graph $\overleftrightarrow{\Gamma}$ of $M$ by $N$ is a forest.
  \end{lemma}
\begin{proof}
Let $s \in M$ and consider the component corresponding to $s$ in $\overleftrightarrow{\Gamma}/N$. Let $w$ be the reduced word that equals $s$. Write
\[
w \equiv w_0 a_1 w_1 \ldots a_m w_m
\]
in Otto-Zhang normal form where the $w_i$ are maximal invertible subwords (possibly empty, and all in $\Delta^*$) and all $a_i \in A$, as in the statement of Lemma~\ref{lem:OttoZhang}.  Let $w'\equiv a_1 a_2 \ldots a_m$.   We claim that the $m$-component of $\overleftrightarrow{\Gamma}/N$ is isomorphic 
to the 
labelled graph denoted 
$\mathcal A(w')$ 
with vertex set $0,\ldots, m$ and an edge labelled by $a_i$ from $i-1$ to $i$ for $1\leq i\leq m$.  It will then follow that  $\overleftrightarrow{\Gamma}/N$ is a forest.

Define a cellular map from the $m$-component of $\overleftrightarrow{\Gamma}$ to $A(w')$ as follows.  If $\alpha_1,\alpha_2\in A^*$ are reduced with $\alpha_1\alpha_2=m$, then $\alpha_1\equiv w_0a_1\cdots w_{i-2}a_{i-1}w_{i-1}'$ and $\alpha_2\equiv w_{i-1}''a_iw_i\cdots a_mw_m$ with $\ov{w_{i-1}'w_{i-1}''}=w_{i-1}$ by Lemmas~\ref{lem:OttoZhang} and~\ref{lemma:OZ}.  Both $w_{i-1}'$ and $w_{i-1}''$ can be empty.  Send $(\alpha_1,\alpha_2)$ to $i-1$.  Notice in particular that the basis element $(w_0a_1\cdots w_{i-2}a_{i-1}, w_{i-1}a_iw_i\cdots a_mw_m)$ is sent to $i-1$, and so this map is surjective on vertices.  By Remark~\ref{remark:useful.basis}, this map factors through $(M\times M)/N$, and the induced map is injective on vertices.    An edge in the $m$-component of  $\overleftrightarrow{\Gamma}$ has the form $(\mu,a,\gamma)$ where $\mu,\gamma\in A^*$ are reduced, $a\in A$ and $\mu a\gamma=m$.  By Lemmas~\ref{lem:OttoZhang} and~\ref{lemma:OZ}, we have two cases.  The first is that $\mu\equiv w_0a_0\cdots w_{i-2}a_{i-1}w_{i-1}'$, $\gamma\equiv w_{i-1}''a_i\cdots a_mw_m$ where $\ov{w_{i-1}'aw_{i-1}''}=w_{i-1}$.  In this case, by Remark~\ref{remark:useful.basis}, this edge is between two vertices in the same weak $N$-orbit, namely that of the basis element $(w_0a_0\cdots w_{i-2}a_{i-1},w_{i-1}a_i\cdots a_mw_m)$.  The second case is that $\mu\equiv w_0a_0\cdots w_{i-1}$, $\gamma\equiv w_i\cdots a_mw_0$ and $a=a_i$.  In this case, initial vertex is in the weak $N$-orbit of the basis element $(w_0a_0\cdots a_{i-1}, w_{i-1}a_i\cdots a_mw_m)$, whereas the terminal vertex is in the weak $N$-orbit of the basis element $(w_0\cdots w_{i-1}a_i, w_i\cdots a_mw_m)$.  We map the corresponding edge of $\overleftrightarrow{\Gamma}/N$ to the edge labeled $a_i$ from $i-1$ to $i$.  This completes the proof of the isomorphism.
\end{proof}

\section{Proving the set of collapsed edges is a finitely generated free $M \times M^{\mathrm{op}}$-set}\label{sec:edge:module}

As in the previous section, throughout this section
\[
M = \lb A \mid w_1=1, \ldots, w_k=1 \rb.
\]

Up to this point we have proved that:

\begin{itemize}
\item
$N \cong G \ast C^*$ for some finite set $C$ where $G$ is the group of units of $M$.
\item
$M \times M$ is a free $N$-set, with basis $\mathcal{B}$ defined above, under the action
\[
(m_1,m_2) \cdot (n_1,n_2) = (m_1 n_1, n_2 m_2).
\]
  \end{itemize}
The final thing we need to prove is that the set $E$ of edges collapsed when forming $\overleftrightarrow{\Gamma}/N$ is a finitely generated free $M \times M^{\mathrm{op}}$-set.
More precisely, we want to consider the edges $E$ of the two-sided Cayley graph $\overleftrightarrow{\Gamma}$ between vertices in the \emph{same} weak $N$-orbit of $M \times M$.
Then $M \times M^{\mathrm{op}}$ acts on this set $E$ of edges via
\[
(m_1, m_2) \cdot (u,a,v) = (m_1 u, a, v m_2)
\]
where $(m_1, m_2) \in M \times M^{\mathrm{op}}$,
$a \in A$ and $u, v, m_1, m_2$ are reduced words over $A$.

Drawing an analogy with our argument in~\cite{GraySteinbergDocumenta} the edges $E$ are the analogue of edges between vertices in the same $\gr$-class in that argument, and $N$ is playing the role of $R_1$ -- the right units.
Note however that $N$ is not in general the submonoid of right units of the monoid $M \times M^{\mathrm{op}}$.

With the above definitions we want to prove the following.

\begin{theorem}\label{thm:E:Gen:Set2}
Let $E$ be the set of edges of the two-sided Cayley graph $\overleftrightarrow{\Gamma}$ between vertices in the \emph{same} weak $N$-orbit.  Then $E$ is a finitely generated free left $M \times M^{\mathrm{op}}$-set under the action
\[
(m_1, m_2) \cdot (u,a,v) = (m_1 u, a, v m_2)
\]
for $(m_1, m_2) \in M \times M^{\mathrm{op}}$ and $(u,a,v) \in E$, with finite basis
\begin{align*}
& \mathcal{C} = \{
(\delta_1, a, \delta_2): \delta_1 a \delta_2\ \text{is invertible},\ a \in A,\ \delta_1,\ \delta_2 \in A^*\ \text{are reduced, and no proper invertible} \\
& \mbox{subword of $\delta_1a\delta_2$ contains the particular occurrence of $a$ that lies between $\delta_1$ and $\delta_2$} \}.
  \end{align*}
Furthermore, if $(\delta_1, a, \delta_2) \in \mathcal{C}$ then $\delta_1 a \delta_2 \in \Delta$.    
\end{theorem}
\begin{proof}
Let $(u,a,v)$ be an edge with $u,v\in A^*$ reduced, and  let \[uav\equiv w_0a_1\cdots a_mw_m\] in Otto-Zhang normal form.  First of all we claim that $(u,a,v)\in E$ if and only if the distinguished occurrence of $a$ is not any of the $a_i$.
 Suppose first that $a_i$ is the distinguished occurrence of $a$ for some $i$, then $uav$ is reduced by Lemma~\ref{lemma:OZ} (since $u,v$ are reduced, and hence their subwords $w_j$ are reduced) and
\begin{align*}
(u,av) & = (w_0a_1\cdots a_{i-1}w_{i-1},a_i\cdots a_mw_m) &  \text{and}\\
(ua,v) &= (w_0\cdots w_{i-1}a_i, w_ia_{i+1}\cdots a_mw_m)
\end{align*}
are in different weak $N$-orbits by Remark~\ref{remark:useful.basis}, as $(u a,v)$ is a basis element and \[(w_0a_1\cdots a_{i-1},w_{i-1}a_i\cdots a_mw_m)\] is the basis element in the weak orbit of $(u,av)$.

On the other hand, suppose that $a$ is a letter of some $w_i$.  Then $u\equiv w_0a_1\cdots a_iw_i'$ and $v\equiv w_i''a_{i+1}\cdots a_mw_m$ with $w_i\equiv w_i'aw_i''$.  It follows from Remark~\ref{remark:useful.basis} that $(ua,v)$ and $(u,av)$ both belong to the weak $N$-orbit of the basis element $(w_0a_1\cdots a_i,\ov{w_i}a_{i+1}\cdots a_mw_m)$.  In particular, each element of $\mathcal C$ belongs to $E$.

We must now show that if $(u,a,v)\in E$, then $(u,a,v) = (m_1,m_2)(\delta_1,a,\delta_2)$ for a unique $(m_1,m_2)\in M\times M^{op}$ and $(\delta_1,a,\delta_2)\in \mathcal C$.  As before let $uav\equiv w_0a_1\cdots a_mw_m$ in Otto-Zhang normal form.  Then, by the above, there is a unique $i$ such that the distinguished $a$ occurs in $w_i$.  Since $w_i$ is invertible, there is a unique minimal invertible subword $\delta_1a\delta_2$ containing this $a$ by Proposition~\ref{p:lattice.vers1}.  Write $u\equiv u'\delta_1$ and $v\equiv \delta_2v'$.  Then $(\delta_1,a,\delta_2)\in \mathcal C$ and $(u,a,v) = (u',v')(\delta_1,a,\delta_2)$.  We claim that these choices are unique.

Clearly the letter $a$ is uniquely determined by $(u,a,v)$.  Suppose that $\alpha,\beta,\delta_1',\delta_2'$ are reduced with $(\alpha,\beta)(\delta_1',a,\delta_2')=(u,a,v)$ and $(\delta_1',a,\delta_2')\in \mathcal C$.  Then $(u,a,v)=(\alpha\delta_1',a,\delta_2'\beta)$.  We claim that $\alpha\delta_1'$ and $\delta_2'\beta$ are reduced.  We handle just the first case as the second is dual.   Since $\alpha,\delta_1'$ are reduced and each left-hand side of a rewrite rule is nonempty and invertible, it follows that the only way we can apply a  rewrite rule is if it overlaps $\alpha$ and $\delta_1'$.   This would imply that $\delta_1'$ has a left invertible prefix.   Now any prefix of $\delta_1'$ is right invertible since $\delta_1'a\delta_2'$ is invertible.   But then we obtain a contradiction to the fact that $\delta_1'$ has no invertible prefix by Proposition~\ref{p:lattice.vers1}.  Thus $\alpha\delta_1'\equiv u'\delta_1$ and, dually, $\delta_2'\beta\equiv \delta_2v'$.  Then $\delta_1'a\delta_2'$ and $\delta_1a\delta_2$ are both minimal invertible subwords of $uav\equiv \alpha\delta_1'a\delta_2'\beta$ containing the distinguished $a$.  By uniqueness of the  minimal invertible subword containing $a$, which follows from Proposition~\ref{p:lattice.vers1}, we conclude that $\delta_1'a\delta_2'\equiv \delta_1a\delta_2$, with the distinguished occurrences of $a$ the same.  Thus $\delta_1\equiv \delta_1'$, $\delta_2\equiv \delta_2'$, $u'\equiv \alpha$ and $v'\equiv \beta$. This establishes the uniqueness, and so $\mathcal C$ is a basis.

To prove that $\mathcal C$ is finite we show that if $(\delta_1,a,\delta_2)\in \mathcal C$, then $\delta_1a\delta_2\in \Delta$.  We begin by showing that $\delta_1a\delta_2\in \Delta^*$.  If $\delta_1a\delta_2$ is reduced, then it belongs to $\Delta^*$ by Lemma~\ref{lem:Zhang:Irred} since it is invertible.  Otherwise, we can apply a rewrite rule to $\delta_1a\delta_2$.  Since $\delta_1,\delta_2$ are reduced, the left-hand side of the rewrite rule must contain $a$.  Since each left-hand side of a rewrite rule is in $\Delta^+$, hence invertible, and $\delta_1a\delta_2$ has no proper invertible subword containing the distinguished $a$, we must have that $\delta_1a\delta_2$ is the left-hand side of a rewrite rule, and hence in $\Delta^*$.  It follows now that there is a subword of $\delta_1a\delta_2$ containing $a$ that belongs to $\Delta$. But then since elements of $\Delta$ are invertible, we deduce by minimality that $\delta_1a\delta_2\in \Delta$.  We conclude that $\mathcal C$ is finite as $\Delta$ is finite.
\end{proof}

\section{Two-sided homological finiteness properties of special monoids}\label{sec:main:theorem:special}

With all the results in place they can be combined to prove the following main result.  We shall use freely that $\mathbb ZM$-bimodules are the same thing as left $\mathbb Z[M\times M^{op}]$-modules.

Before giving the proof, we make a general topological remark.  If $X$ is a CW complex, then the cellular chain  complex can be augmented to \[\cdots\rightarrow C_1(X)\to C_0(X)\to H_0(X)\to 0\] where the map $C_0(X)\to H_0(X)$ is the projection.   This chain complex is natural with respect to cellular maps.  Identifying $H_0(X)$ with the free abelian group on the path components of $X$,  this map sends each vertex (viewed as a basis element of $C_0(X)$) to its path component.  Hence we can view this complex as the direct sum of the standard augmented cellular chain complexes of the path components of $X$.   In particular, it will be exact if each path component of $X$ is contractible. 

\begin{theorem}[Theorem~A]\label{thm:main:special:monoids2}
Let $M$ be a finitely presented monoid of the form 
\[M = \lb A \mid w_1=1, \ldots, w_k=1 \rb\]
and let $G$ be the group of units of $M$.
\begin{enumerate}
\item[(i)]
If $G$ is of type $\FPn$ (for $1 \leq n \leq \infty$) then $M$ is of type bi-$
\FPn$.
\item[(ii)]
Furthermore, the Hochschild cohomological dimension of $M$ is bounded below by
$\mathrm{cd}(G)$
and bounded above by
$\mathrm{max}\{ 2, \mathrm{cd}(G) \}$.
  \end{enumerate}
\end{theorem}
\begin{proof}
Let $\overleftrightarrow{\Gamma}$ be the two-sided Cayley graph of $M$ and let $F$ be  $\overleftrightarrow{\Gamma} /N$ which is a forest by Lemma~\ref{lem:forst:quotient2}.
Notice that the multiplication map $M \times M \rightarrow M$ factors through $(M \times M)/N$ and hence the connected components of $F$ are still isomorphic to $M$ as an $M$-biset.  The augmented cellular chain complex
\[
0 \rightarrow C_1(F) \rightarrow C_0(F) \rightarrow H_0(F)\cong \ZM \rightarrow 0
\]
is a bimodule resolution of $\ZM$ with exactness following from the discussion before the theorem, as $F$ is a forest and hence has contractible path components.

Now $C_1(F)$ is the quotient of a free bimodule on $A$, namely $\mathbb Z[M\times A\times M]$, by $\mathbb ZE$ where $E$ is the finitely generated free $M\times M^{op}$-set from Theorem~\ref{thm:E:Gen:Set2}.  Thus we have a free resolution
 \[0\longrightarrow \mathbb ZE\longrightarrow \mathbb Z[M\times A\times M]\to  C_1(F)\to 0\] by finitely generated free bimodules, and so $C_1(F)$ is $\FPinfty$ and has
projective dimension at most 1.

We have $C_0(F)\cong \Z[(M\times M)/N]\cong \Z[M \times M]\otimes_{\ZN} \Z$.
In Theorem~\ref{thm:FreeProductStructure} we proved that $N$ is a free product of the group of units $G$ and
a finitely generated free monoid.
Together with~\cite[Corollary 4.13]{GraySteinbergDocumenta} this then implies that
$N$ is of type left-$\FPn$, and hence $\Z$ is a left $\ZN$-module of type $\FPn$. Let $F_\bullet\to \Z\to 0$ be a free resolution of $\Z$ by $\ZN$-modules with $F_i$ finitely generated for $0\leq i\leq n$.  Then since $M\times M^{op}$ is a free right $N$-set, it follows that $\Z[M \times M^{op}]$ is free as a right $\ZN$-module and hence flat.  Thus $\mathbb Z[M\times M^{op}]\otimes_{\ZN} F_\bullet \to  \mathbb Z[M\times M^{op}]\otimes_{\ZN} \Z\to 0$ is a free resolution with $\mathbb Z[M\times M^{op}]\otimes_{\ZN} F_i$ finitely generated for $0\leq i\leq n$.  Therefore, $C_0(F)\cong \Z[M \times M^{op}]\otimes_{\ZN} \Z$ is a $\mathbb Z[M\times M^{op}]$-module of type $\FPn$.
Now since $C_0(F)$ is $\FPn$ and $C_1(F)$ is $\FPinfty$ it follows from Lemma~\ref{t:bieri}
that the $\ZM$ is $\FPn$ as a $\mathbb Z[M\times M^{op}]$-module, and hence the monoid $M$ is of type bi-$\FPn$.

By~\cite[Corollary 4.15]{GraySteinbergDocumenta}, since $N$ is a free product of $G$ and a free monoid
we can choose a free resolution of $\Z$ over $\ZN$ to have length at most
$\mathrm{max}\{1, \mathrm{cd}(G) \}$,
and so we can deduce from the above argument that $C_0(F)$ has projective dimension at most $\max\{1,\mathrm{cd}(G)\}$.   Since $C_1(F)$ has projective dimension at most $1$, we deduce that $\ZM$ has projective dimension at most
$\mathrm{max}\{2, \mathrm{cd}(G) \}$,
and thus $M$ has Hochschild dimension at most $\mathrm{max}\{2, \mathrm{cd}(G)\}$, by Lemma~\ref{c:fp.resolved}.

The fact that the Hochschild cohomological dimension of $M$ is bounded below by that of $G$ follows from~\cite[Theorem 3.16(2)]{GraySteinbergDocumenta} together with the fact that the Hochschild cohomological dimension bounds both the left and right cohomological dimension of a monoid; see~\cite[Section~7]{GraySteinbergAGT}.   
\end{proof}

\subsection{Applying the general result to special one-relator monoids}

In the one-relator case the main theorem above specialises to the following which gives a Lyndon's Identity type theorem for the two-sided homology of one-relator monoids.

\begin{theorem}[Theorem~B]\label{thm:main:one:relator}
Every one-relator monoid of the form $\langle A \mid r=1 \rangle$ is of type bi-$\FP_\infty$. Moreover, if $r$ is not a proper power then the one-relator monoid has Hochschild cohomological dimension at most $2$, while if $r$ is a proper power then the monoid has infinite Hochschild cohomological dimension.
  \end{theorem}
\begin{proof}
Let $M= \langle A \mid r=1 \rangle$. The group of units $G$ of $M$ is a one-relator group by Adjan's theorem~\cite[Lemma 96]{Adjan1966} and hence the group $G$ is of type $\FP_\infty$ by Lyndon's Identity Theorem~\cite{Lyndon1950}. Now applying Theorem~\ref{thm:main:special:monoids2}(i) we deduce that $M$ is of type bi-$\FP_\infty$.

For the second statement, if $r$ is not a proper power then the group of units $G$ will be  a torsion-free one-relator group (see e.g.~\cite[Lemma~3.18]{GraySteinbergDocumenta}). Then Lyndon's results~\cite{Lyndon1950} imply that $G$ has cohomological dimension at most 2, which implies the same for the Hochschild cohomological dimension   of $M$ by Theorem~\ref{thm:main:special:monoids2}(ii).

If $r$ is a proper power then $G$ has torsion and hence infinite cohomological dimension.  Therefore $M$ has infinite  Hochschild cohomological dimension   by Theorem~\ref{thm:main:special:monoids2}(ii).
\end{proof}
\section{Other one-relator monoids}\label{sec:other}

In~\cite{GraySteinbergSelecta} we proved that all one-relator monoids $\lb A \mid u=v \rb$ are of type left-$\FP_\infty$ (and dually of type right-$\FP_\infty$). In that paper we also classified the one-relator monoids of cohomological dimension at most 2. It is natural to ask to what extent the results in this paper on two-sided homology might be extended to general one-relator monoids $\lb A \mid u=v \rb$. In particular, it would be interesting to see whether the techniques in this paper can be further developed to prove that all one-relator monoids $\lb A \mid u=v \rb$ are of type bi-$\FP_\infty$. 

One natural approach to answering this question would be to try to use the same general proof strategy that was employed in~\cite{GraySteinbergSelecta} in the one-sided case. From that viewpoint, the result in this paper showing that special monoids are all of type bi-$\FP_\infty$ would provide one important infinite family of base cases. The other infinite family of base cases are the, so-called, \emph{strictly aspherical} one-relator monoids. 
It turns out (as we shall explain below) that known results can be applied to deduce that these strictly aspherical one-relator monoids are also of type bi-$\FP_\infty$. Hence combined with the results in this paper that deals with all the base cases. Thus following same approach as in~\cite{GraySteinbergSelecta} this means that the key to proving that all one-relator monoids are of type bi-$\FP_\infty$ will be to develop an understanding of how the property bi-$\FP_\infty$ behaves under a process called compression (a common framework for both Adjan–Oganesyan compression~\cite{Adyan1987} and Lallement compression~\cite{Lallement1974} that was developed in~\cite{GraySteinbergSelecta}). Developing this understanding in the one-sided case was already quite an involved task, occupying much of the paper~\cite{GraySteinbergSelecta}, and we expect the two-sided case to be even more challenging.  

Let us outline this approach in more detail. For full details of the terminology used here we refer the reader to~\cite{GraySteinbergSelecta}. Let $M$ be the one-relator monoid defined by the presentation $\lb A \mid u=v \rb$ where, without loss of generality, we assume $|v| \leq |u|$. This presentation for $M$ is called \emph{compressible} if there is a nonempty word $r \in A^+$ such that $u,v \in A^*r \cap rA^* $. Otherwise, the presentation is called incompressible. If $u,v \in A^*r \cap rA^* $ then there is an associated compressed one-relator monoid $M_r$ that has a shorter defining relation, and the word problem of $M_r$ is equivalent to that of $M$. The compression operation is transitive and confluent and it follows that there is a unique incompressible monoid $M'$ to which $M$ compresses. 
      
It follows from~\cite[Lemma 3.4 and Proposition 3.5]{GraySteinbergSelecta} 
and~\cite[Corollary 5.6]{Kobayashi1998} 
 that if $M$ is the one-relator monoid defined by the presentation $\lb A \mid u=v \rb$ and $M'$ is the unique incompressible monoid to which $M$ compresses then either 
\begin{enumerate} 
\item[(i)] $M'$ is a special one-relator monoid $\lb B \mid w=1 \rb$, or else 
\item[(ii)] $M'$ is a non-special one-relator monoid with a strictly aspherical presentation $\lb B \mid w=z \rb$. 
  \end{enumerate}
Here a monoid presentation is strictly aspherical if each connected component of the Squier complex of the presentation is simply connected (see~\cite[Section~2]{GraySteinbergSelecta} for the relevant definitions).

In case (i) above it follows from the results in this paper that $M' = \lb B \mid w=1 \rb$
is of type bi-$\FP_\infty$.   
In case (ii) the compressed monoid $M' = \lb B \mid w=z \rb$ is strictly aspherical, and it turns out, as we shall now explain, that this is enough to show that 
$M'$
is also of type bi-$\FP_\infty$.   
Indeed, a key fact used in the proof of the main result of~\cite{GraySteinbergSelecta} was that the Cayley complex of a strictly aspherical presentation turns out to be contractible; see~\cite[Lemmas 6.4 and 6.5]{GraySteinbergSelecta}.  It turns out that key fact also holds in the two-sided case too, and that it is likely to be equally important in the approach we are describing here to proving that all one-relator monoids are of type bi-$\FP_\infty$. 

In more detail, in~\cite{Pride1995} Pride defines $\pi_2^{(b)}(\mathcal{P}) = H_1(D(\mathcal{P}))$ where $D(\mathcal{P})$ is the Squier complex of the presentation.  Hence if $\mathcal{P}$ is strictly aspherical then $\pi_1(D(\mathcal{P}))$ is trivial at every vertex, hence in this case $\pi_2^{(b)}(\mathcal{P}) = H_1(D(\mathcal{P})) = 0$.  Then the argument in the final two paragraphs of the article~\cite{Steinberg_2025} shows that the second homology of the 2-sided Cayley 2-complex of the presentation $\mathcal{P}$ is equal to $\pi_2^{(b)}(\mathcal{P})$.  When the presentation $\mathcal{P}$ is strictly aspherical this implies that second homology of the two-sided Cayley complex is $0$. Together with the fact proved in~\cite[Theorem 7.12]{GraySteinbergAGT} that each connected component of the two-sided Cayley complex is simply connected, 
this implies 
by the Hurewicz and Whitehead theorems
that the  connected components of the two-sided Cayley complex are contractible. 
Then it follows that the two-sided Cayley graph is a bi-equivariant classifying space for the monoid, in the sense defined in~\cite[Section~7]{GraySteinbergAGT}, from which it follows that the monoid is bi-$\F_\infty$ and hence by~\cite[Proposition 7.7]{GraySteinbergAGT} is also bi-$\FP_\infty$.    
Also this shows that a strictly aspherical finitely presented monoid has Hochschild cohomological dimension at most $2$, again by results of~\cite[Section~7]{GraySteinbergAGT}, since the argument above shows that the 2-sided Cayley 2-complex is a bi-equivariant classifying space for the monoid.   

In particular the observations in the previous paragraph can be used to give an alternative proof that when $w$ is not a proper power the torsion-free special one-relator monoid $\langle A \mid w=1 \rb$ is of type bi-$\FP_\infty$ and has Hochschild cohomological dimension at most $2$, since Kobayashi proved in~\cite[Corollary~7.5]{Kobayashi2000} that torsion-free special one-relator monoids are strictly aspherical. In addition these observations show the property bi-$\FP_\infty$ holds for one-relator monoids also in the following cases. 

\begin{Thm} 
Let $M = \lb A \mid u=v \rb$ be an incompressible one-relator monoid with $u,v \in A^+$, that is, there is no nonempty word $r \in A^+$ such that $u,v \in A^*r \cap r A^*$. Then $M$ is of type bi-$\FP_\infty$ and has Hochschild cohomological dimension at most $2$.   
\end{Thm}
The class of examples covered by the previous proposition should, from the discussion in this section above, be viewed as a family of base cases for a possible future inductive proof using compression that all one-relator monoids are of type bi-$\FP_\infty$.

\end{document}